\documentclass[10pt,reqno]{amsart}

\usepackage{amsmath,amscd,amssymb,amsthm}
\usepackage{tikz-cd}
\usepackage{comment}

\makeatletter
\def\th@plain{%
  \thm@notefont{}
  \itshape 
}
\def\th@definition{%
  \thm@notefont{}
  \normalfont 
}
\makeatother

\usepackage{enumitem}
\setlist[enumerate]{label=(\roman*),leftmargin=0.8cm}
\usepackage{todonotes}
\usepackage{color}

\newtheorem{proposition}{Proposition}[section]
\newtheorem{lemma}[proposition]{Lemma}
\newtheorem{theorem}[proposition]{Theorem}
\newtheorem{corollary}[proposition]{Corollary}

\theoremstyle{definition}
\newtheorem{remark}[proposition]{Remark}
\newtheorem{definition}[proposition]{Definition}
\newtheorem{examples}[proposition]{Examples}

\numberwithin{equation}{section} 

\newcommand{\R}{\mathbb{R}}
\newcommand{\C}{\mathbb{C}}

\newcommand{\Z}{\mathbb{Z}}
\newcommand{\Q}{\mathbb{Q}}

\newcommand{\pr}{\mathbb{P}}

\newcommand{\ddbar}{\partial\bar{\partial}}

\newcommand{\B}{\mathcal{B}}
\newcommand{\scL}{\mathcal{L}}

\newcommand{\scH}{\mathcal{H}}
\newcommand{\scM}{\mathcal{M}}

\newcommand{\scX}{\mathcal{X}}
\newcommand{\X}{\mathcal{X}}
\newcommand{\Y}{\mathcal{Y}}

\newcommand{\scO}{\mathcal{O}}

\newcommand{\scY}{\mathcal{Y}}
\newcommand{\scN}{\mathcal{N}}
\newcommand{\scZ}{\mathcal{Z}}

\newcommand{\co}{\Omega}

\newcommand{\reg}{\text{reg}}
\newcommand{\sing}{\text{sing}}

\DeclareMathOperator{\Bl}{Bl}
\DeclareMathOperator{\Aut}{Aut}

\DeclareMathOperator{\Ric}{Ric}
\DeclareMathOperator{\Rel}{Rel}

\DeclareMathOperator{\DF}{DF}
\DeclareMathOperator{\Amp}{Amp}

\DeclareMathOperator{\AM}{AM}

\title[K-stability for K\"ahler Manifolds]{K-stability for K\"ahler manifolds}

\author[Ruadha\'i Dervan and Julius Ross]{Ruadha\'i Dervan and Julius Ross}
\address{Ruadha\'i Dervan, DPMMS, Centre for Mathematical Sciences, Wilberforce Road, Cambridge CB3 0WB, United Kingdom, and Universit\'e Libre de Bruxelles, Franklin Rooseveltlaan 50, 1050 Brussels, Belgium.}
\email{R.Dervan@dpmms.cam.ac.uk}

\address{Julius Ross, DPMMS, Centre for Mathematical Sciences, Wilberforce Road, Cambridge CB3 0WB, United Kingdom.}
\email{J.Ross@dpmms.cam.ac.uk}

\begin{document}

\begin{abstract} We formulate a notion of K-stability for K\"ahler manifolds, and prove one direction of the Yau-Tian-Donaldson conjecture in this setting.  More precisely, we prove that the Mabuchi functional being bounded below
(resp. coercive) implies K-semistability (resp.\ uniformly K-stable).  In particular this shows that the existence of a constant scalar curvature K\"ahler metric implies K-semistability, and K-stability if one assumes the automorphism group is discrete.  We also show how Stoppa's argument holds in the K\"ahler case, giving a simpler proof of this K-stability statement.
\end{abstract}

\maketitle

\section{Introduction}

The search for canonical metrics in K\"ahler geometry has led to the important notion of K-stability.    Originally defined by Tian \cite{Tian97} in the context of K\"ahler-Einstein metrics on Fano manifolds and by analogy with the Hilbert-Mumford criterion in Geometric Invariant Theory, K-stability involves control of the sign of a numerical invariant associated to certain one-parameter degenerations of the original manifold. This numerical invariant is essentially the same quantity discovered by Futaki \cite{Futaki} which is given as an integral over the central fibre of the degeneration, and has since been given a purely algebro-geometric interpretation by Donaldson \cite{DonaldsonToric}.  For this reason the invariant is commonly referred to as the \emph{Donaldson-Futaki} invariant, and extends the notion of K-stability to projective manifolds that are not necessarily Fano.

Over time it has become clear that it is useful to think of the Donaldson-Futaki invariant as a kind of topological invariant of the total space of a given degeneration. This point of view was exploited first by Wang \cite{WangHeight} and Odaka \cite{OdakaGIT,Odakaslope}, and is key to the relationship between K-stability and birational geometry, in particular the work of Li-Xu \cite{LiXu}, as well as the non-Archimidean viewpoint taken up by Boucksom-Hisamoto-Jonsson \cite{BHJ}.

The point of this paper is to emphasise that the topological definition of the Donaldson-Futaki invariant allows one to define K-stability even more generally, and applies to K\"ahler manifolds that are not necessarily projective.   This extension appears in embryonic form in \cite{BermanTemperature,RossThomas,StoppaSlope,Sz2}. Our main result is a precise definition of the Donaldson-Futaki invariant and K-stability in this context as well as a proof of the so-called ``easy direction" of the Yau-Tian-Donaldson conjecture in this setting.  In the following let $X$ be a compact complex manifold and $[\omega]$ be a K\"ahler class on $X$. 
\begin{theorem}\label{thm1}
Suppose the Mabuchi functional for $[\omega]$ is bounded (resp.\ coercive).  Then $(X,[\omega])$ is K-semistable (resp.\ uniformly K-stable).
\end{theorem}

A deep analytic result of Berman-Berndtsson \cite{BB} (resp.\ Darvas-Rubinstein \cite{Dar-Rub} and Berman-Darvas-Lu \cite{BermanDarvasLu}) says that the Mabuchi functional is bounded  (resp.\ coercive when $\Aut(X,[\omega])$ is discrete) if one assumes that $[\omega]$ admits a constant scalar curvature K\"ahler (cscK) metric.   So combining this with the above we get one direction of the Yau-Tian-Donaldson conjecture in this setting (when this project started the result of \cite{BermanDarvasLu} was not available, so we were only able to conclude K-semistability).

In the projective case a clever argument of Stoppa \cite{Stoppablowup} proves the existence of a constant scalar curvature K\"ahler metric implies K-stability (also under the assumption there are no infinitesimal automorphisms) using a blowup technique and the glueing theorem of Arezzo-Pacard \cite{ArezzoPacard}.      Our second result shows that this technique also works in the K\"ahler case (and moreover gives a slightly different proof of Stoppa's theorem in the projective case).

\begin{theorem}\label{thm2}
 Assume the automorphism group of $(X,[\omega])$ is discrete.  If there exists a constant scalar curvature K\"ahler metric in $[\omega]$ then the pair $(X,[\omega])$ is K-stable.
\end{theorem}

\newcommand{\stab}{\operatorname{stab}}

\subsection*{Discussion} The depth of the Yau-Tian-Donaldson conjecture in the projective case lies in its linking analysis (through the K\"ahler-Einstein or cscK equation) and finite dimensional algebraic geometry (through stability).    By the projectivity assumption it is unclear what, if anything, replaces this finite dimensional picture,  but it is still interesting to ask if there are any features of the projective case that survive.  For instance, one can ask how K-stability behaves in families; that is if $Y\to S$ is a flat family of manifolds what can be said about the set $S_{\stab} = \{s\in S : Y_s \text{ is K-stable}\}$?   For this to make sense in the projective case one  needs to assume also the data of a relatively ample $\mathcal L\to Y$ making each $Y_s$ polarised.  When this polarisation is the anti-canonical bundle (so each $Y_s$ is Fano), a corollary of the Chen-Donaldson-Sun  \cite{CDSI,CDSII,CDSIII} techniques gives that $S_{\stab}$ is Zariski-open \cite{DonaldsonOpen,Odakamoduli}.   In the non-Fano case, it may be that Zariski-openness is too much to ask for, but it seems likely that $S_{\stab}$ is at least the complement of a countable number of algebraic sets in $S$.   In the non-projective case $\mathcal L$ must be replaced by a suitable $(1,1)$-class, and we ask if $S_{\stab}$ has an analogous property.  More precisely, is $S_{\stab}$ the complement of a countable number of analytic subsets of $S$?  We remark that there exist analogous statement for stability for vector bundles (or sheaves) on K\"ahler manifolds (see \cite{teleman} or \cite[Prop 2.9,Theorem 11.6]{GRTI}).  This, along with results of Hong \cite{Hong}, suggest the above is true for particular families of ruled manifolds, but other than this the question seems very much open.

\subsection*{Outline} We start in Section \ref{sec:Kstabilitydefinition} by defining the Donaldson-Futaki invariant and K-stability in a way that immediately extends to the non-projective case.   After recalling some of the basic functionals that we shall need in Section \ref{sec:mabuchi-prelims}, we prove Theorem \ref{thm1} in Section \ref{sec:semistability}.   Then in Section \ref{sec:Stoppa} we give a modification of the blowup argument of Stoppa \cite{Stoppablowup} giving Theorem \ref{thm2}.  Finally in Section \ref{sec:related} we give extensions that apply to twisted constant scalar curvature K\"ahler metrics and to the J-flow.

\subsection*{Acknowledgements}  As this work was in progress we learned of similar results of Sj\"ostr\"om Dyrefelt \cite{Dyrefelt} who independently proves Theorem \ref{thm1} with a slightly different method (see Remark \ref{rmk:Dyrefelt}).  We are very grateful to both Sebastien Boucksom and Zakarias Sj\"ostr\"om Dyrefelt for  helpful conversations on this and related topics.  We would also like to thank Henri Guenancia, Yuji Odaka and Jacopo Stoppa,as well as the referees for their valuable comments.   Part of this work was done while the first author visited the Simons Center for Geometry and Physics, the first author thanks the SCGP for the stimulating environment. JR is supported by an EPSRC Career Accelleration Grant (EP/J002062/1) which also provides RD's studentship.  RD has received additional support from a Fondation Wiener-Anspach scholarship.

\subsection*{Notations}
A polarised variety is a pair $(X,L)$ where $X$ is a variety and $L$ an ample line bundle on $X$.  To simplify notation if $L_1,\ldots,L_n$ are line bundles on $X$ where $n=\dim X$ we write $L_1\ldots L_n$ to mean the intersection $c_1(L_1)\cdots c_1(L_n)$.    We will exclusively deal with normal varieties that admit a canonical (Weil) divisor $K_X$, and we let $L_1\ldots L_{n-1}. K_X$ denote the intersection $c_1(L_1)\cdots c_1(L_{n-1}).[K_X]$. 

When $X$ is smooth, the Dolbeaut cohomology class of a $(1,1)$ form (or current) $\omega$ will be denoted by $[\omega]$.  If $[\omega_1],\ldots,[\omega_n]$ are $(1,1)$-classes we let $[\omega_1]\cdots[\omega_n] = \int_X \omega_1\wedge \cdots \wedge \omega_n$. 

Given a family $\pi\colon \X\to B$ for $t\in B$ we let $\X_t:=\pi^{-1}(t)$ and if $\scL$ is a line bundle on $\X$ let $\scL_t: = \scL|_{\X_t}$.   Similarly if $\Omega$ is a form on $\X$ we let $\Omega_t:= \Omega|_{\X_t}$.  We denote the projections from a product $X_1\times X_2$ to the two factors by either $p_i$ for $i=1,2$ (or $p_{X_i}$ for $i=1,2$).  Finally if $p:X\to Y$ is a morphism and $\omega$ is a form on $Y$ we will sometimes write $\omega$ for $p^*\omega$ when no confusion is possible.

\section{Definition of K-stability for K\"ahler manifolds}\label{sec:Kstabilitydefinition}

In this section we make a precise definition of K-stability for K\"ahler manifolds.  To put this in context we start with a presentation of the definition in the projective case.

\subsection{K-stability in the projective case}

Let $(X,L)$ be a normal polarised variety of dimension $n$.

\begin{definition}[Test-configuration, projective case]\label{def:testconfig} A \emph{test-configuration} for $(X,L)$ is a normal polarised variety $(\X,\scL)$ together with 
\begin{enumerate}
\item a $\C^*$-action on $\X$ lifting to $\scL$,
\item a flat $\C^*$-equivariant map  $\pi\colon \X\to \mathbb P^1$ where $\mathbb P^1$ is given the standard $\C^*$-action,
\end{enumerate}
such that $(\pi^{-1}(\mathbb P^1\setminus\{0\}),\mathcal L)$ is $\mathbb C^*$-equivariantly isomorphic to the product $(X\times \mathbb C^*, p_X^*L^{\otimes r})$ (where the latter is given the trivial action on the $X$ factor).  The number $r$ is called the \emph{exponent} of the test-configuration.\end{definition}

\begin{definition}[Slope of a polarised variety, projective case]\label{def:slope}
We define the \emph{slope} of $(X,L)$ to be $$\mu(X,L) := \frac{-K_X.L^{n-1}}{L^n} = \frac{-\int_X c_1(K_X). c_1(L)^{n-1}}{\int_X c_1(L)^n}.$$
\end{definition}

\begin{definition}[Donaldson-Futaki invariant, projective case]\label{def:DFprojective} The \emph{Donal\-dson-Futaki invariant} of a test-configuration $(\X,\scL)$ for $(X,L)$ of exponent $r$ is 
 $$\DF(\X,\scL) := \frac{n}{n+1}\mu(X,L^{\otimes r}) \scL^{n+1} + \scL^n.K_{\X / \pr^1}.$$

 \end{definition}

\begin{remark}
\begin{enumerate}
\item As we are assuming $\X$ is normal, the assumption that $\pi$ is flat in the definition of a test-configuration is automatically satisfied \cite[III 9.7]{Hartshorne}.
\item Again because of normality,  $\mathcal X$ has a canonical divisor $K_{\mathcal X}$ which is a Weil-divisor, and thus also a relative canonical divisor $K_{\mathcal X/\mathbb P^1} = K_{\mathcal X} - \pi^* K_{\mathbb P^1}$ making the intersection $\scL^n.K_{\mathcal X/\mathbb P^1}$ well-defined.
\item We can extend the notion of test-configuration, Donaldson-Futaki invariant and K-stability to the case that $\scL$ and $L$ are $\mathbb R$-line bundles in the obvious way. In particular, we can and do assume that the exponent of the test-configuration is one.  
\end{enumerate}
\end{remark}

Another important concept is the  \emph{minimum norm} of a test-configuration \cite{BHJ,Dervan}. To define this observe that every test-configuration admits a birational map $$f:(X\times\pr^1,p_1^*L)\dashrightarrow (\X,\scL),$$ so one can take a resolution of indeterminacy as follows.

\[
\begin{tikzcd}
\Y \arrow[swap]{d}{q} \arrow{dr}{g} &  \\
X\times\pr^1 \arrow[dotted]{r}{f} & \X
\end{tikzcd}
\]

\begin{definition}[Minimum norm of a test-configuration, projective case]We define the \emph{minimum norm} of $(\X,\scL)$ to be $$\|(\X,\scL)\|_m = g^*\scL.(q^*L)^n - \frac{(g^*\scL)^{n+1}}{n+1}.$$
\end{definition}
\begin{proposition} The minimum norm of $(\mathcal X,\mathcal L)$ is non-negative.  Moreover a test-confi\-guration is isomorphic to $(X\times\pr^1,L)$ if and only if its minimum norm is zero. \end{proposition}

\begin{proof}
The second statement is proved in \cite[Theorem 1.3]{Dervan} (and the non-negativity of the norm is proved along the way) and independently in \cite{BHJ}.   We invite the reader to compare with similar results in \cite{LejmiSz}. 
\end{proof}

\begin{definition}[K-stability for projective varieties]We say a polarised variety $(X,L)$ is \begin{enumerate}\item \emph{K-semistable} if $\DF(\X,\scL)\geq 0$ for all test-configurations $(\X,\scL)$ for $(X,L)$, 
\item \emph{K-stable} if $\DF(\X,\scL)>0$ for all test-configurations $(\X,\Omega)$ for $(X,L)$ with $\|(\X,\scL)\|_m>0$, \item \emph{uniformly K-stable} if there exists an $\epsilon>0$ such that $\DF(\X,\co)\ge\epsilon\|(\X,\scL)\|_m $ for all test-configurations $(\X,\scL)$ for $(X,L)$.\end{enumerate}\end{definition}

\begin{remark}\begin{enumerate}
\item The definition above of the minimum norm varies slightly to that given in \cite{Dervan}. Namely, the definition here is equal to the non-Archimedean $J$-functional as defined in \cite{BHJ}, while the definition in \cite{Dervan} corresponds to the non-Archimedean $I-J$-functional. Analogous to Lemma \ref{IJarenonneg}, it is proved in \cite{BHJ} 
that this does not affect the definition of (uniform) K-stability. 
\item The above definitions differ slightly from elsewhere in the literature.  One can instead consider a test-configuration as a flat family over $\mathbb C$ with the same properties.    However this changes nothing as any such family is isomorphic to to $(X\times\C^*,p_1^*L)$ away from $t=0$ so one can glue in the trivial family around infinity to compactify to give a test-configuration in the above sense.  Furthermore we assume that the total space $\X$ is normal, but this is known not to affect the definition of K-semistability or uniform K-stability \cite[Proposition 5.1]{RossThomas} (and also does not affect the definition of K-stability as we assume in (ii) that $\|(\mathcal X,\mathcal L)\|_m$ is strictly positive). Moreover the usual definition of K-stability often requires $\scL$ to be \emph{relatively} ample, again this is easily seen to be equivalent to our definition.  Finally Odaka and Wang \cite{Odakaslope,WangHeight} both prove that the definition of the Donaldson-Futaki invariant given here agrees with the definition of Donaldson \cite{DonaldsonToric} that is given in terms of the weight of the induced $\C^*$-action on the section ring of the central fibre of $(\X,\scL)$, see also \cite{LiXu}.   This, in turn, is equivalent to Futaki's original invariant (as originally used by Tian \cite{Tian97}) when the central fiber is normal \cite{DonaldsonToric}.
\item In the Fano case the most important of these is K-stability, since it has been proved by Chen-Donaldson-Sun \cite{CDSI,CDSII,CDSIII} that a Fano manifold admits a K\"ahler-Einstein metric if and only if $(X,-K_X)$ is K-stable.  For general polarisations it is strongly expected that something more than K-stability is needed to guarantee the existence of a canonical K\"ahler metric (such as one with constant scalar curvature), and uniform K-stability is a candidate for this \cite{BHJ,Dervan}.  In \cite{Sz1} Sz\'ekelyhidi proposes a beautiful notion of K-stability using filtrations that can be thought of as a form of uniform K-stability restricted to a particular set of sequences of test-configurations.
\item In the projective case, it is known that the existence of a canonical K\"ahler metric implies K-stability \cite{bermankpoly, DonaldsonCalabi, Stoppablowup, Tian97}, see also \cite{BermanDarvasLu,CS, StoppaSz} for related results in the presence of automorphisms.
\item The definition of K-stability is unchanged if one restricts to test-configur\-ations whose central fibre $\mathcal X_{0}:=\pi^{-1}(0)$ is reduced (that is, one demands only that the relevant inequality for the Donaldson-Futaki invariant holds for this restricted class of test-configurations).  In fact, one can perform a base change and then normalisation to obtain from any test-configuration a new test-configuration with reduced central fibre, whilst controlling both the Donaldson-Futaki invariant \cite[Proposition 7.14, Proposition 7.15]{BHJ} and the minimum norm \cite[Remark 7.11, Prop 7.23]{BHJ}.  We prove the analytic counterpart of this statement in Proposition \ref{prop:simplifiedtestconfig}.
\end{enumerate}
\end{remark}

\subsection{K-stability in the non-projective case}

To define K-stability in the non-projective case we will use the notion of a smooth K\"ahler form on an analytic space.   To discuss this we first need the notion of forms on such a space \cite{Demailly}, so let $X$ be a complex analytic space that is reduced and of pure dimension $n$.  We let $X_{\reg}$ and $X_{\sing}$ denote respectively the regular and singular locus of $X$.  Roughly, a smooth form on $X$ is defined locally as the restriction of a smooth form from some embedding of $X$ into affine space.  That is, if $j\colon X\to \Omega$ is a local embedding where $\Omega\subset \mathbb C^N$ is open, a $(p,q)$-form on $X$ is defined as the image of the restriction map
$$j^*\colon \mathcal A^{p,q}(\Omega) \to \mathcal A^{p,q}(X_{\reg}).$$
  That this is well-defined is verified in \cite[p14]{Demailly} and comes from the fact that given any other local embedding $j'\colon X\to \Omega'\subset \mathbb C^{N'}$ there exists (locally) holomorphic $f\colon \Omega\to \mathbb C^{N'}$ and $g\colon \Omega'\to \mathbb C^N$ such that $j'=fj$ and $j=gj'$, from which one can check that the image of $j^*$ and $j'^*$ agree. In particular this definition agrees with the usual definition of smooth forms on the smooth locus $X_{\reg}$.   The exterior derivative, wedge product are defined in the obvious way on the regular locus.   If $F\colon X\to Y$ is a morphism between analytic spaces then one can define a pullback $F^*$ of forms by considering local embeddings $j\colon X\to \Omega$ and $j'\colon Y\to \Omega'$ such that $F$ lifts to a map $\Omega\to \Omega'$ and pulling back from $\Omega'$ to $\Omega$.   That this is well defined is verified in  \cite[Lemma 1.3]{Demailly} (the only subtlety being when $F(X)$ is contained in the singular locus of $Y$).

\begin{definition}[K\"ahler space]
A \emph{K\"ahler space} $(X,\omega)$ is an analytic space $X$ and a smooth $(1,1)$-form $\omega$ on $X$ that is locally the restriction of a smooth K\"ahler form under an embedding of $X$ into an open subset of some $\mathbb C^N$.  We refer to $\omega$ as a \emph{smooth K\"ahler form} on $X$, and observe that for the same reason as above this notion is well-defined.   
\end{definition}

From now on we let $X$ be a compact K\"ahler manifold of dimension $n$ and $\omega$ a K\"ahler form with cohomology class $[\omega]$. 

\begin{definition}[Test-configuration] A test-configuration for $(X,[\omega])$ is a normal K\"ahler space $(\X,\co)$, together with
\begin{enumerate}
\item a surjective flat map $\pi: \X\to\pr^1$,
\item a $\C^*$-action on $\X$ covering the usual action on $\pr^1$ such that $\Omega$ is $S^1$-invariant and so that the $\C^*$ action preserves the Bott-Chern cohomology class of $\Omega$,
\item A biholomorphism $\alpha:\pi^{-1}(\mathbb P^1\setminus \{0\})\simeq X\times \mathbb P^{1}\setminus\{0\}$ that is $\mathbb C^*$-equivariant such that for all $t\in \mathbb P^1\setminus \{0\}$ we have $[\Omega_t] = [\alpha_t^*\omega]$ as cohomology classes.
\end{enumerate}
\end{definition}

\begin{examples} As the above definition of a test configuration is new, we give some simple examples.  Note that if $X$ is in fact projective then one class of test configurations for $(X,[\omega])$ can be obtained by taking a non-integral K\"ahler form $\Omega$ on a usual (i.e. projective) test configuration for $X$.    In the following we will see there are genuinely non-projective examples.

\begin{enumerate}
\item (Products)  If a K\"ahler manifold $(X,[\omega])$ admits a holomorphic $\mathbb C^*$-action, then the product $X \times \C$ admits an induced action. One can compactify to a family $\pi:\X\to\pr^1$,  which admits a relatively K\"ahler class  $\alpha$ induced from $[\omega]$. Picking an $S^1$-invariant $\co\in\alpha$ gives a test configuration with respect to $\co+\pi^*\omega_{FS}$, called a product test configuration.

\item (Degeneration to normal cone) Suppose $Y\subset X$ is a submanifold of a compact complex manifold $X$, and let $p:\mathcal X\to X\times \mathbb P^1$ be the blowup of $Y\times \{0\}$ inside $X\times \mathbb P^1$.  Then the product $\mathbb C^*$-action on $X\times \mathbb P^1$ that acts trivially on the $X$ factor lifts to $\mathcal X$.  Letting $[E]$ denote the cohomology class of the exceptional divisor, for sufficiently small $c>0$ the cohomology class $p^*[\omega]- c[E]$ admits a K\"ahler form $\Omega$ whose restriction to $1\in \mathbb P^1$ is $\omega$ \cite[Lemma 3.4]{Ver}, and by averaging over the $S^1$-action we may assume $\Omega$ is $S^1$-invariant.   Thus $(\mathcal X,\co)$ is a test configuration for $(X,[\omega])$.   K-stability with respect to test-configurations constructed in this way gives the notion of slope-stability of $X$ with respect to $Y$ which was first studied in the projective case by Ross-Thomas \cite{RossThomas,RossThomas2}  and in the analytic case by Stoppa \cite{StoppaSlope}.    The reader will find further (non-projective) examples in \cite[Sec 5]{StoppaSlope}, including a slope-unstable K\"ahler manifold that is not deformation equivalent to any projective manifold.
\item (Toric test-configurations)  The well known correspondence between toric manifolds and polytopes gives rise to test configurations that are themselves toric.   Suppose that $(X,\omega)$ is a toric K\"ahler manifold and $P$ the image of the moment polytope.  We recall that when $\omega$ is an integral class, $P$ will be a lattice polytope, but there are examples of toric K\"ahler manifolds that are not projective \cite[p84]{Oda}.   Let $f: P\to \mathbb R$ be a concave strictly positive piecewise linear function.  Then the polytope $Q = \{ (x,t) \in P\times \mathbb R: t\le f(x)\}$ is a polytope of one dimension higher than $P$, that gives rise to a toric $\mathcal X$ that is the total space of a test-configuration for $(X,[\omega])$.    This idea been studied in detail by Donaldson (e.g. \cite{DonaldsonToric}) who emphasises the projective case, but much of what is written does not require this hypothesis.
\item (Projective bundles) Let $(B,\omega_B)$ be a K\"ahler manifold and $E$ be a hermitian holomorphic vector bundle on $B$.   Set $X = \mathbb P(E)$.  Then the hermitian metric on $E$ induces a $(1,1)$-form on $X$ that is positive in the fibre directions, and so pulling back by a suitable multiple of $\omega_B$ gives a K\"ahler form $\omega$ on $X$.   We can then form a test configuration for $(X,[\omega])$ starting with any degeneration of $E$ to a vector bundle $E_0$, by thinking of such a degeneration as a holomorphic vector bundle $\mathcal E$ over $B\times \mathbb C$ and letting $\mathcal X = \mathbb P(\mathcal E)$ over $\mathbb C$ and then glueing the trivial family over $\pr^1\setminus\{0\}$.
\item (Limits of submanifolds) Suppose $(P,\omega_P)$ is a complete non-projective K\"ahler manifold that admits a holomorphic $\mathbb C^*$-action preserving $[\omega_P]$   (for example one may take $P$ to be the projectivization of a vector bundle that admits a $\mathbb C^*$-action).  Let $X\subset P$ be a submanifold and $\omega:= \omega_P|_X$.    Moving $X$ with the $\mathbb C^*$-action gives a family over $\mathbb C^*$ of submanifolds of $P$ which we may think of as a morphism from $\mathbb C^*$ to the Douady space of $P$.  Since $P$ is a compact K\"ahler manifold, a theorem of Fujiki \cite[Theorem 4.5]{Fujiki} allows one to complete this to get a family $\mathcal X\subset \mathbb C\times P$ which is flat over $\mathbb C$.  Normalising and then glueing the trivial family over $\pr^1\setminus\{0\}$ gives a test-configuration for $(X,[\omega])$. 
 \end{enumerate}
\end{examples}
\color{black}

\begin{definition}[Slope of a K\"ahler manifold]The \emph{slope} of a K\"ahler manifold is defined to be $$\mu(X,\omega) := \frac{c_1(X).[\omega]^{n-1}}{[\omega]^{n}}.$$\end{definition}

For test-configurations with smooth total space we can define the Donaldson-Futaki invariant by complete analogy to the projective case:

\begin{definition}[Donaldson-Futaki invariant I]\label{def:DFI}Let $(\pi:\X\to \mathbb P^1,\co)$ be a test-confi\-guration for $(X,[\omega])$  with $\X$ smooth. We define the \emph{Donaldson-Futaki invariant} of a $(1,1)$-form $\Omega'$ on $\mathcal X$ to be
\begin{equation*}
\DF(\X,\Omega') :=  \frac{n}{n+1}\mu(X,\omega) [\Omega']^{n+1} - (c_1(\X) - \pi^*c_1(\pr^1)).[\Omega']^n.\label{eq:DFI}\end{equation*}
The \emph{Donaldson-Futaki invariant} of $(\mathcal X, \Omega)$ is defined to be $DF(\mathcal X,\Omega)$.
\end{definition}

 When the test-configuration is singular, we take a resolution of singularities $p:\scY\to\X$. By Hironaka's Theorem we can, and will, assume that $p$ is an isomorphism away from the central fibre of $\X$, so that there are equivariant isomorphisms $$\scY\backslash \scY_0\cong\scX\backslash \scX_0\cong X\times\C.$$ Moreover we assume that $\Y$ is constructed from $\X$ by successive blowups along smooth centres, and we will always assume our resolutions are equivariant. Observe that the semi-positive form $p^*\co$ may not be strictly positive on $\Y$ (and so is not a K\"ahler form).

\begin{definition}[Donaldson-Futaki invariant II]
Let $(\X,\Omega)$ be a test-config\-uration for $(X,\omega)$ and $p:\scY\to\X$ be a resolution.  We define the \emph{Donaldson-Futaki} invariant of $(\X,\Omega)$ to be
  \begin{align*}
  \DF(\X,\co) &:=  DF(\mathcal Y,p^*\Omega)  \\
  &=\frac{n}{n+1}\mu(X,\omega) [p^*\co]^{n+1} - (c_1(\Y) - (\pi\circ p)^*c_1(\pr^1)).[p^*\co]^n.\end{align*}
  \end{definition}

\begin{lemma}\label{lem:IndOfResSing} The Donaldson-Futaki invariant is independent of resolution. 

\end{lemma}
\begin{proof}

Let $\mathcal Y'\stackrel{q}{\to} \mathcal Y \stackrel{p}{\to} \mathcal X$ be a tower of resolutions.  Then $q^* c_1(\mathcal Y)  = c_1(\mathcal Y') + [D]$ where $D$ is a sum of divisors that are exceptional for $q$.  So $[D].[q^*p^* \Omega]^n=0$ since $q$ maps each component of $D$ to something of dimension at most $n-1$.  In fact if $D =\sum a_i D_i$ then $[D_i][q^*p^*\Omega]^n = \int_{D_i} q^*p^*\Omega^n = \int_{D_i} \iota^* q^*p^* \Omega^n$ where $\iota:D_i\to \mathcal Y'$ is the inclusion.  But $p\circ \iota$ factors though some submanifold $Z$ of $\mathcal Y$ and $q^*p^*\Omega^n|_Z=0$ for dimension reasons; compare \cite[7.2]{DemaillyMonge}.  Using this one sees immediately that
$$DF(\mathcal Y,p^*\Omega) = DF(\mathcal Y',q^*p^* \Omega).$$
Since any two resolutions of $\mathcal X$ are dominated by a third, this proves the lemma.
\end{proof}

\begin{remark} To see that the term $[p^*\co^{n+1}]$ is independent of resolution one can also simply note that $[p^*\co^{n+1}] = \int_{\X_{reg}} \co^{n+1}$.\end{remark}

\color{black}

\begin{remark}\label{RicciCurrent}

One could instead attempt to define the Donaldson-Futaki invariant on a test-configuration with normal total space by making sense of the intersection products in Definition \ref{def:DFI}.  For example \cite{Dyrefelt} does this using the intersection in Bott-Chern cohomology. We have instead chosen to define the Donaldson-Futaki invariant using a resolution of singularities, which is equivalent and changes rather little for us, since in the end we want to work on such a resolution.
\end{remark}

As in the projective case we also need a notion of a norm of a test-configuration. Let $(\X,\co)$ be a test-configuration for  $(X,[\omega])$. Denote by $$f: X\times\pr^1 \dashrightarrow \X$$ the natural bimeromorphic map. We can take a resolution of indeterminacy as follows. \[
\begin{tikzcd}
\Y \arrow[swap]{d}{q} \arrow{dr}{g} &  \\
X\times\pr^1 \arrow[dotted]{r}{f} & \X
\end{tikzcd}
\]

\begin{definition}[Minimum norm of a test-configuration] \label{def:minnormanalytic} Let $(\mathcal X,\Omega)$ be a test-configuration for $(X,[\omega])$.   We define the \emph{minimum norm} of a $(1,1)$-form $\Omega'$ on $\mathcal X$ to be $$\|(\X,\Omega')\|_m = [g^*\Omega'].[q^*\omega]^n - \frac{[g^*\Omega']^{n+1}}{n+1}.$$
The \emph{minimum norm} of $(\mathcal X,\Omega)$ is defined to be $\|(\X,\Omega)\|_m$.
\end{definition}

Just as in Lemma \ref{lem:IndOfResSing}, we have the following. 

\begin{lemma}\label{canworkonres} The minimum norm is independent of choice resolution of indeterminacy. \end{lemma}

We will prove the following in Proposition \ref{min-norm-is-nonneg-proof}. 

\begin{proposition}\label{MinimumNormIsNonnegative} The minimum norm of a test-configuration is non-negative.\end{proposition}

\begin{remark}\label{minimum-norm-using-co_1} By the definition of a test configuration there is a biholomorphism $f:X\cong \X_1$ such that $[f^*\Omega_1] = [\omega]$.  So abusing notation slightly, the minimum norm is therefore also given as $$\|(\X,\co)\|_m = [g^*\co].[q^*\co_1]^n - \frac{[g^*\co]^{n+1}}{n+1}.$$\end{remark}

We are now ready to define the notions of stability relevant to us.

\begin{definition}[K-stability for K\"ahler manifolds] We say a K\"ahler manifold $(X,[\omega])$
 is \begin{enumerate} \item \emph{K-semistable} if $\DF(\X,\co)\geq 0$ for all test-configurations $(\X,\Omega)$ for $(X,[\omega])$, 
\item \emph{K-stable} if $\DF(\X,\co)>0$ for all test-configurations $(\X,\Omega)$ for $(X,[\omega])$ with $\|(\X,\co)\|_m>0$, \item \emph{uniformly K-stable} if there exists an $\epsilon>0$ such that $\DF(\X,\co)\ge\epsilon\|(\X,\co)\|_m $ for all test-configurations $(\X,\Omega)$ for $(X,[\omega])$.\end{enumerate}\end{definition}

\
The following proposition says that for K-semistability and uniformly K-stability it is sufficient to consider only test-configurations with smooth total space and whose central fibre is reduced.

\begin{proposition}\label{prop:simplifiedtestconfig}
Let $(X,\omega)$ be a compact K\"ahler manifold.  Then $(X,[\omega])$ is K-semistable (resp.\ uniformly K-stable) if $DF(\mathcal X,\Omega)\ge 0$ (resp.\ there exists an $\epsilon>0$ such that $DF(\mathcal X,\Omega)\ge \epsilon \|(\mathcal X,\Omega)\|_m$) for all test-configurations $(\mathcal X,\Omega)$  for $(X,[\omega])$ such that $\mathcal X$ is smooth and whose central fibre $\mathcal X_0$ is reduced. 

Moreover one can assume the natural bimeromorphic map $\mathcal X\dashrightarrow X\times\pr^1$ is defined on all of $\mathcal X$, so no resolution of indeterminacy is needed in the definition of the minimum norm.
\end{proposition}
\begin{proof}
(1) We first show how to reduce to test-configurations with smooth total space.  Let $(\mathcal X,\Omega)$ be a test-configuration for $(X,[\omega])$.   It is shown in \cite[Lemma 2.2]{CMM} that there exists a resolution of singularities $p: \mathcal Y\to \mathcal X$ such that $\mathcal Y$ has a K\"ahler metric $\zeta$ in the cohomology class $p^*[\Omega] - c [E]$ where $E$ is the exceptional divisor in $\mathcal Y$ and $c$ is some sufficiently small positive real number.    We observe the reference \cite{CMM} applies to the resolution of singularities constructed in \cite{BM}, which is canonical and so the $\mathbb C^*$-action on $\mathcal X$ lifts to $\mathcal Y$; see \cite[Theorem 13.2(2)]{BM}.    Averaging $\zeta$ over the induced $S^1$-action we may assume it is $S^1$-invariant. 

Now, for small $\delta>0$ let
$$\Omega_\delta := (1+\delta)^{-1}( p^* \Omega + \delta \zeta)$$
which is an $S^1$-invariant K\"ahler form on $\mathcal Y$ making $(\mathcal Y,\Omega_\delta)$ a test configuration for $(X,[\omega])$ with smooth total space.  By Lemma \ref{lem:IndOfResSing}, the Donaldson-Futaki invariant of $(\mathcal X,\Omega)$ can be calculated on any resolution, giving
\begin{align*}
DF(\mathcal X,\Omega) &= DF(\mathcal Y,p^*\Omega) = \frac{n}{n+1}\mu(X,\omega) [p^*\co]^{n+1} - (c_1(\mathcal Y)- (\pi p)^*c_1(\pr^1)).[p^*\co]^n\\
&= DF(\mathcal Y,\Omega_\delta) + O(\delta).
\end{align*}
A similar calculation allows one to compare the minimum norm of $(\mathcal X,\Omega)$ and $(\mathcal Y,\Omega_\delta)$.  In detail let $f:X\times\mathbb P^1\dashrightarrow  \mathcal X$ be the natural bimeromorphic map which lifts to a map $\tilde{f}:X\times\mathbb P^1\dashrightarrow \mathcal Y$ so that $p\tilde{f} = f$.

Taking a resolution of indeterminancy of $\tilde{f}$ gives a smooth $\mathcal Z$ fitting in a diagram
\[\begin{tikzcd}
\mathcal Z \arrow{r}{g}\arrow{d}{q}  & \mathcal Y \arrow[swap]{d}{p}   \\
X\times\pr^1 \arrow[dotted]{ur}{\tilde{f}} \arrow[dotted]{r}{f} & \X
\end{tikzcd}
\]
Then $\tilde{g}p:\mathcal Z\to \mathcal X$ is a resolution of indeterminancy of $f$, and so by Lemma \ref{canworkonres} it can be used to calculate both the minimum norm of $(\mathcal X,\Omega)$ and of $(\mathcal Y,\Omega_\delta)$.  Thus
\begin{align*}
 \|(\mathcal Y,\Omega_\delta)\|_m &= [g^* \Omega_\delta] [q^*\omega]^n- \frac{[g^*\Omega_\delta]^{n+1}}{n+1}= [g^* p^*\Omega] [q^*\omega]^n + \frac{[g^*p^*\Omega]^{n+1}}{n+1} + O(\delta)\\
 & = \|(\mathcal X,\Omega)\|_m + O(\delta).
 \end{align*}
This is enough to prove the first statement, for if $(X,[\omega])$ is not K-semistable then there is a test-configuration $(\mathcal X,\Omega)$ for $(X,[\omega])$ with $DF(\mathcal X,\Omega)<0$ and so taking $\delta$ sufficiently small $DF(\mathcal Y,\Omega_\delta)<0$ as well.  Analogous arguments work for uniform K-stability, and the statement regarding the existence of a surjective map to $X\times\pr^1$.

(2) We next show how to reduce to test-configurations with reduced central fibre (this is similar to the proof in the projective case \cite[Section 7.3]{BHJ}).   Let $(\mathcal X,\Omega)$ be a test-configuration for $(X,[\omega])$ with smooth total space.  Consider the map $u:\mathbb P^1\to \mathbb P^1$ given by $z\mapsto z^d$ and let $\mathcal X'$ be the normalisation of the pullback of $\mathcal X$ along $u$.  We denote the induced finite map $\mathcal X'\to \mathcal X$ also by $u$.  Setting $\Omega':= u^*\Omega$ we have that $(\mathcal X',\Omega')$ is a test-configuration for $(X,[\omega])$ which, for $d$ sufficiently large and divisible, has reduced central fibre \cite[Section 16]{Kollar-NX}.  Thus we need to compare the Donaldson-Futaki invariant and minimum norms of these two test-configurations.   

We start with the minimum norm.   Consider a resolution of indeterminancy
\begin{equation}\label{indet1}
\begin{tikzcd}
\Y \arrow[swap]{d}{q} \arrow{dr}{g} &  \\
X\times\pr^1 \arrow[dotted]{r}{f} & \X
\end{tikzcd}
\end{equation}

Let $\mathcal Y'$ be the normalisation of the pullback of $\mathcal Y$ along $u$, and denote the induced finite map also by $u:\mathcal Y'\to \mathcal Y$. Thus we have a diagram
 \begin{equation}
\begin{tikzcd}\label{indet2}
\Y' \arrow[swap]{d}{q'} \arrow{dr}{g'} &  \\
X\times\pr^1 \arrow[dotted]{r}{f'} & \X'
\end{tikzcd}
\end{equation}
such that the maps $u\colon X\times\mathbb P^1\to X\times\mathbb P^1$, $u:\mathcal X'\to \mathcal X$ and $u:\mathcal Y'\to \mathcal Y$ take \eqref{indet2} to \eqref{indet1} (that is, the obvious diagram commutes).     Then let $p:\mathcal Z'\to \mathcal Y'$ be a resolution, so $\mathcal Z'\to \mathcal X'$ is a resolution of indeterminancy for $X\times\mathbb P^1\dashrightarrow \mathcal X'$.  Using this to compute the minimum norms gives

\begin{align}\label{eq:minimumnormhom}
\|(\mathcal X',\Omega')\|_m &= [p^*g'^*\Omega'][p^*q'^* \Omega_1']^n - \frac{[p^*g'^* \Omega']^{n+1}}{n+1} \nonumber\\
 &= [g'^*\Omega'][q'^* \Omega_1']^n - \frac{[g'^* \Omega']^{n+1}}{n+1} \nonumber\\
&= [u^* g^* \Omega] [ q^* u^* \Omega_1]^n - \frac{[u^* g^* \Omega]^{n+1}}{n+1}\nonumber\\
&= d \|(\mathcal X,\Omega)\|_m.
\end{align}

We now turn to the Donaldson-Futaki invariant.  Using the pullback formula for the canonical bundle \cite[2.41.4]{Kollar} under finite maps (which still holds in the analytic setting \cite[Section 2.1, p38]{Kollar}), we have
$$ c_1(\mathcal Y') - p'^* \pi^* c_1(\mathbb P^1) = u^* (c_1(\mathcal Y) - p^*\pi^* c_1(\mathbb P^1)) + [R]$$
for some effective divisor $R$ in $\mathcal Y'$ (this calculation is exactly as in the projective case, for which we refer the reader to \cite[equation (4.6)]{BHJ}).  Using $[R].[g'^*\Omega']^n\ge 0$ gives
\begin{align*}
DF(\mathcal X',\Omega') &= DF(\mathcal Z', p^*g'^*\Omega) \\
&= \frac{n}{n+1} \mu(X,\omega) [p^*g'^*\Omega']^{n+1} - (c_1(\mathcal Z') - p^*q'^*\pi^* c_1(\mathbb P^1)) [g'^*\Omega']^n \\
&= \frac{n}{n+1} \mu(X,\omega) [g'^*\Omega']^{n+1} - (c_1(\mathcal Y') - q'^*\pi^* c_1(\mathbb P^1)) [g'^*\Omega']^n \\
&\le  \frac{n}{n+1} \mu(X,\omega) [u^*g^*\Omega]^{n+1} - (u^* c_1(\mathcal Y) -  u^* q^*\pi^* c_1(\mathbb P^1)).[u^*g^*\Omega]^n\\
&= d\left( \frac{n}{n+1} \mu(X,\omega) [g^*\Omega]^{n+1} - (c_1(\mathcal Y) -  p^*\pi^* c_1(\mathbb P^1).[g^*\Omega]^n\right)\\
&= d DF(\mathcal Y,g^*\Omega)\\
&= d DF(\mathcal X,\Omega)
\end{align*}
So in total we have
\begin{equation*}
 DF(\mathcal X',\Omega') \le d DF(\mathcal X,\Omega)
\end{equation*}
Together with \eqref{eq:minimumnormhom} this completes the proof.

\end{proof}
\color{black}

\begin{lemma}\label{normalisingtheDF} For all $c\in\R$ and all smooth $(1,1)$-forms $\eta\in c_1(\scO_{\pr^1}(1))$  we have $$\DF(\X,\co+c\pi^*\eta) = \DF(\X,\co).$$ 
and
$$ \|(\mathcal X,\co + c\pi^*\eta)\|_m = \|(\mathcal X,\co)\|_m.$$
\end{lemma}

\begin{proof}Without loss of generality we may assume $\mathcal X$ is smooth.   We show $$[\co+c\pi^*\eta]^{n+1}=[\co]^{n+1}+c(n+1)[\Omega_t]^n,$$ and $$[\co+c\pi^*\eta]^n.[c_1(\X) - \pi^*c_1(\pr^1)]=[\co]^n.[c_1(\X)-\pi^*c_1(\pr^1)]+(cn)c_1(\X_t).[\co_t]^{n-1},$$for some (or equivalently all) $t$, which imply the result. 

Note that $\pi_*\co^{n}$ is smooth away from $0$, since $\pi$ is a submerssion on this locus \cite[Lemma 2.15]{JPD-book}. From Poincar\'e-Lelong it then follows that for $t\neq 0$ we have \begin{align*}\int_{\X} \co^{n}\wedge(\pi^*\eta) &= \int_{\X\backslash\X_0} \co^{n}\wedge(\pi^*\eta), \\ &= \int_{\pr^1\backslash 0} \pi_*(\co^{n})\wedge\eta, \\ &= \int_{\pr^1\backslash 0} \pi_*(\co^{n})\wedge \{t\}, \\ &= \int_{\X_t} \co_t^n.\end{align*} The push-pull formula implies $$[\co]^{n+1-i}.[\pi^*\eta]^i = 0$$ for all $i\geq 2$, which then gives the first required equation. 

The adjunction formula gives $c_1(\X).\X_t = c_1(\X_t),$ since the intersection of the fibre with itself is trivial. It follows that $$[\co+c\pi^*\eta]^n.(c_1(\X)-\pi^*c_1(\pr^1)) = [\co]^n.(c_1(\X)-\pi^*c_1(\pr^1)) + (cn)[\co_t]^{n-1}.c_1(\X_t),$$ as required.  The second statement is proved similarly.
\end{proof}

\section{Preliminaries on the Mabuchi functional}\label{sec:mabuchi-prelims}

Let $(X,\omega)$ be a compact K\"ahler manifold. The $\ddbar$-lemma implies any other K\"ahler metric in the K\"ahler class $[\omega]$ can be written as $\omega+i\ddbar \varphi$, for some $\varphi\in C^{\infty}(X,\R)$. 

\begin{definition}[Mabuchi functional]\cite{Mabuchi-functional} Fix a path $\varphi_t \in C^{\infty}(X,\R)$ in the space of K\"ahler potentials with $\varphi_0 = 0$ and $\varphi_1=\varphi$ and corresponding K\"ahler metrics $\omega_t$. We define the \emph{Mabuchi functional} to be $$\scM(\varphi) := -\int_0^1\int_X \dot\varphi_t (S(\omega_t)-n\mu(X,\omega))\omega_t^n,$$ where $S(\omega_t)$ denotes the scalar curvature of $\omega_t$. \end{definition}

As the notation suggests, the Mabuchi functional is independent of chosen path. The Mabuchi functional also admits an explicit formulation as follows, due to Chen \cite{Chen-IMRN} and Tian \cite[Section 7.2]{Tian-Book}. 

\begin{align*}
\scM(\varphi)& = \int_X \log\left(\frac{\omega_{\varphi}^n}{\omega^n}\right)\omega_{\varphi}^n + \frac{n}{n+1}\mu(X,[\omega]) \sum_{i=0}^n \int_X \varphi \omega^i\wedge\omega_{\varphi}^{n-i} \\
&- \sum_{i=0}^{n-1} \int_X \varphi \Ric\omega\wedge \omega^i\wedge\omega_{\varphi}^{n-1-i}.\end{align*}

We will also require some more functionals.

\begin{definition}[Aubin-Mabuchi, J and I functionals]\label{functionals-def} Let $(X,\omega)$ be a K\"ahler manifold, and let $\varphi$ be a K\"ahler potential. We define
\begin{enumerate}
\item the \emph{Aubin-Mabuchi energy} of $\varphi$ as $$\AM(\varphi) = \sum_{i=0}^n \int_X \varphi \omega^i\wedge\omega_{\varphi}^{n-i},$$

\item the $J$\emph{-functional} of $\varphi$ as $$J(\varphi) := \int_X \varphi \omega^n - \frac{\AM(\varphi)}{n+1},$$
\item the $I$\emph{-functional} of $\varphi$ as $$I(\varphi):=\int_X\varphi (\omega^n-\omega_{\varphi}^n).$$\end{enumerate}\end{definition}

We will later use the following standard properties of the $I$ and $J$-functionals.

\begin{lemma}\label{IJarenonneg}\cite[Lemma 6.19, Remark 6.20]{Tian-Book} The $I$ and $J$-functionals are non-negative. Moreover, $$\frac{1}{n}J(\varphi)\leq I(\varphi)-J(\varphi)\leq nJ(\varphi),$$\end{lemma}

\begin{definition}[Coercivity]\cite[Section 7.2]{Tian-Book} We say the Mabuchi functional is \emph{coercive} if $$\scM(\varphi) \geq \epsilon J(\varphi) + c$$ for some $\epsilon>0$ independent of $\varphi$ and some $c\in \R$.\end{definition}

Remark that one could equivalently use $I(\varphi)$ in the definition of coercivity, by Lemma \ref{IJarenonneg}. Moreover, by Lemma \ref{IJarenonneg}, coercivity of the Mabuchi functional implies it is bounded from below.   It will be useful to introduce a final functional.

\begin{definition}[$L_{\alpha}$-functional]\label{L-alpha} Fix an arbitrary smooth $(1,1)$-form $\alpha$ on $X$. We define  $$L_{\alpha}(\varphi) := \sum_{i=0}^n\int_X \varphi \alpha\wedge\omega^i\wedge\omega_{\varphi}^{n-1-i}.$$\end{definition}

One notes that the Mabuchi functional is then given as $$\scM(\varphi) = \int_X \log\left(\frac{\omega_{\varphi}^n}{\omega^n}\right)\omega_{\varphi}^n + \frac{n}{n+1}\mu(X,[\omega]) \AM(\varphi) - L_{\Ric\omega}(\varphi).$$

\section{The Mabuchi Functional and K-stability}\label{sec:semistability}

In this section we introduce certain currents on $\pr^1$, arising from a given test-configuration. These allow us to relate the analytic functionals (Mabuchi functional, J-functional, Aubin-Mabuchi energy) to the corresponding intersection numbers. The upshot will be the following (see Corollary \ref{cor:K-semistability} and Corollary \ref{cor:Kstable}).

\begin{theorem}\label{thm1:repeat}
Suppose the Mabuchi function for $[\omega]$ is bounded below (resp.\ coercive).  Then $(X,[\omega])$ is K-semistable (resp.\ uniformly K-stable).  
\end{theorem}

Our approach is closely related to arguments in the projective case, originally due to Tian \cite{Tian97} and built on by Paul-Tian \cite{Paul-Tian}, Phong-Ross-Sturm \cite{PRS} and Berman \cite{bermankpoly} (among others).

\begin{remark} \label{rmk:Dyrefelt} As mentioned in the introduction Sj\"ostr\"om Dyrefelt \cite{Dyrefelt} has also proved this using a slightly different method.  In fact \cite{Dyrefelt} gives more, namely an interpretation using intersection theoretic quantities of the limit derivative of the Mabuchi functional along a path given by a test-config\-uration with possibly non-reduced central fibre. We only consider test-configurations with smooth total space and reduced central fibre, which is sufficient for our purposes following Proposition \ref{prop:simplifiedtestconfig}.
\end{remark}

Our constructions will be of the following flavour. Let $(\X,\co)$ be a test-confi\-guration. Since the morphism $\pi: \X\to\pr^1$ is proper, given an arbitrary current on $\X$, we obtain a current on $\pr^1$ by taking the direct image. We apply this to the top degree forms appearing the the definition of the Donaldson-Futaki invariant and the minimum norm. The corresponding direct image currents will be smooth on $\C^*$, and hence will equal $i\ddbar$ of some smooth function. The goal of this section is to determine that function. 

To set this up, recall the $\mathbb C^*$-action on $\mathcal X$ gives for each $t\in \mathbb C^*$ a biholomorphism $\rho(t):\mathcal X\to \mathcal X$.  By abuse of notation we denote also by $\rho(t)$ its restriction $\rho(t): \X_1 \cong \X_t$.   By hypothesis each $\rho(t)$ preserves the cohomology class of $\Omega$.  We can then choose a smooth family $\zeta_t\in C^{\infty}(\mathcal X)$ such that
$$\rho(t)^* \Omega - \Omega = i\ddbar \zeta_t.$$
\begin{definition}(Normalisation of potentials)\label{def:normalisationofpotentials}
Let
$$\varphi_t:=\zeta_t|_{\X_1} $$
so
$$\rho(t)^*\co_t - \co_1  = i\ddbar \varphi_t.$$
\end{definition}
\color{black}
The key point is that $\rho(t)$ is not an isometry but $\rho(t)^*\Omega_t$ and $\Omega_1$ lie in the same cohomology class on $\mathcal X_1$. Now let $\mathcal X^{\times} = \pi^{-1}(\mathbb P^1\setminus \{0\})$ and continue to let $\pi\colon \mathcal X^{\times} \to \mathbb P^1\setminus \{0\}$ be the projection.   We recall the definition of a test-configuration assumes the existence of a $\mathbb C^*$-equivariant biholomorphism 
 \begin{equation}\label{eq:notationfortriv}
 \alpha: \mathcal X^\times \to X\times \mathbb P^1\setminus\{0\}
 \end{equation} 
 such that $\alpha|_{\X_t}^*[\omega] = [\Omega_t]$ for $t\in \mathbb P^1\setminus \{0\}$.    We then use $\alpha|_{\mathcal X_1}$ to identify $\mathcal X_1$ with $X$. 
 
\begin{definition} 
 Given any $p$-form $\zeta$ on $X$ we set
$$\rho(t)_*\zeta : = \alpha^* p_1^*\zeta$$
where $p_1:X\times \mathbb P^1\setminus \{0\} \to X$ is the projection.
\end{definition} 
Thus $\rho(t)_*\zeta$ is a $p$-form on $\mathcal X^{\times}$ obtained by moving the form $\zeta$ from the fibre $\mathcal X_1\simeq X$ to $\mathcal X_t$ using the biholomorphism $\rho(t):\mathcal X_1\simeq \mathcal X_t$. 
\color{black}

As a first step, we now consider the Aubin-Mabuchi energy. In the notation of Definition \ref{functionals-def}, $\co_1$ plays the role of $\omega$, with $\rho(t)^*\co_t$ playing the role of $\omega_{\varphi}$.

\begin{proposition}\label{AMasddbar}Let $(\X,\co)$ be a test-configuration for $(X,[\omega])$, and let $t$ be the usual coordinate on $\pr^1$. Then the pushforward $\pi_*\co^{n+1}$ is smooth away from $t=0$, and is given as $$\pi_*\co^{n+1} = i\ddbar \AM(\varphi_t).$$ Moreover it gives $\{0\}$ zero measure. \end{proposition}

\color{black}

\begin{proof}
As the test-configuration $\X$ is equivariantly isomorphic to $\X\times\C$ away from $0\in\pr^1$, the morphism $\pi$ is a submersion on this locus. In particular, it is smooth on this locus by properties of the the direct image \cite[Section 2.15]{JPD-book}. 

The mass given by $\pi_*\co^{n+1}$ to $\{0\}$ is calculated as \begin{align*}\int_{\{0\}} \pi_*\co^{n+1} &= \int_{\X_0} \co^{n+1}, \\ &=0\end{align*}
which holds as $\Omega$ is smooth over the total space $\mathcal X$.

We wish to show $$\pi_*\co^{n+1}= i\ddbar \int_{X}\varphi_t\left(\sum^n_{i=0} \co_1^i \wedge \rho(t)^*\co_t^{n-i}\right).$$ Pushing forward by $\rho(t)$ gives $$\int_{X}\varphi_t\left(\sum \co_1^i \wedge \rho(t)^*\co_t^{n-i}\right) = \int_{\X_t}(\rho(t)_*\varphi_t)\left(\sum^n_{i=0} (\rho(t)_*\co_1^i )\wedge \co_t^{n-i}\right).$$

With $f$ a test function (in particular, of compact support) and $U \subset\C^*$ open, we have

$$\int_{U}f \pi_*\co^{n+1} = \int_{\pi^{-1}(U)}\pi^*f\co^{n+1}.$$ On the other hand, we have 

\begin{align*}\int_U f i\ddbar &\int_{X}\varphi_t\left(\sum^n_{i=0} \co_1^i \wedge \rho(t)^*\co_t^{n-i}\right) =  \int_U (i\ddbar f)  \int_{X}\varphi_t\left(\sum^n_{i=0} \co_1^i \wedge \rho(t)^*\co_t^{n-i}\right), \\ &= \int_U (i\ddbar f)\int_{\X_t}(\rho(t)_*\varphi_t)\left(\sum^n_{i=0} (\rho(t)_*\co_1^i) \wedge \co_t^{n-i}\right),  \\ &= \int_{\pi^{-1}(U)} (\rho(t)_*\varphi_t)(i\ddbar \pi^*f)\wedge \left(\sum^n_{i=0} (\rho(t)_*\co_1^i) \wedge \co^{n-i}\right), \\ &= \int_{\pi^{-1}(U)} \pi^*f (\rho(t)_*(i\ddbar \varphi_t))\wedge\left(\sum^n_{i=0} (\rho(t)_*\co_1^i) \wedge \co^{n-i}\right).\end{align*} 

Remark that $\rho(t)_*(i\ddbar \varphi_t) =\co - \rho(t)_*\co_1$, while $$(\co - \rho(t)_*\co_1)\wedge\left(\sum_{i=0}^n (\rho(t)_*\co_1^i)\wedge \co^{n-i}\right) = \co^{n+1} - \rho(t)_*\co_1^{n+1} .$$ This gives the result since $\co_1^{n+1}=0$.
\end{proof}

Before considering the more complicated functionals, we will need a relative form of the Ricci curvature.    To define this, suppose that $\tau$ is an $(n,n)$-form on a complex manifold $X$ of dimension $n$, and that $\tau$ is not identically zero.  Then $\tau$ is a section of $K_X\otimes \overline{K}_X$ which we may then think of as a (singular) metric $h_{\tau}$ on $-K_X$.

The curvature of this metric is a current, which by standard abuse of notation we denote by
$$-i\ddbar \log \tau:=-i\ddbar \log h_{\tau} \in c_1(-K_X).$$
When $\tau = \omega^n$ for some K\"ahler metric $\omega$ then $\tau$ is a volume form, $h_{\tau}$ is smooth, and its curvature is just the Ricci form of $\omega$,
$$\Ric(\omega) = -i\ddbar  \log h_{\omega^n} = -i\ddbar \log \omega^n\in - c_1(K_X).$$

We require a relative version of this construction.  Suppose $\pi:\mathcal X\to \mathbb P^1$ is a holomorphic map with $\mathcal X$ smooth of dimension $n+1$, and let $\omega_{FS}$ be the Fubini-Study metric on $\mathbb P^1$.    We suppose that for all $t\neq 0$ the fibre $\mathcal X_t = \pi^{-1}(t)$ is smooth of dimension $n$.

\begin{definition}
Given a semi-positive $(1,1)$-form $\Omega$ on $\mathcal X$  let
$$\Rel(\Omega) := -i\ddbar \log(\Omega^n\wedge \pi^*\omega_{FS})-2\pi^*\omega_{FS}.$$
\end{definition}
In this definition we implicitly assume that $\Omega^n\wedge\pi^*\omega_{FS}$ is not identically zero, which will always be the case below.   Thus the cohomology class of $[\Rel(\Omega)]$ is 
$$[\Rel(\Omega)] = -c_1(K_{\mathcal X}) + c_1(K_{\mathbb P^1}) = -c_1(K_{\mathcal X/\mathbb P^1}).$$ 

\begin{lemma}\label{lem:ricmove}
Over $\mathcal X^\times$ we have
$$\Rel(\rho(t)_*\Omega_1)= \rho(t)_*\Ric(\Omega_1)$$
\end{lemma}
\begin{proof}
Recall $h_{\tau}$ denotes the hermitian metric on the anticanonical bundle induced by a top-degree form $\tau$ and $\alpha:\mathcal X^\times \simeq X\times \mathbb P^1\setminus\{0\}$ is the biholomorphism from \eqref{eq:notationfortriv}.  We claim that under the isomorphism
$$p_1^*K_X\otimes \pi^* K_{\mathbb P^1\setminus \{0\}}= K_{X\times \mathbb P^1\setminus\{0\}} \stackrel{\alpha^*}{\simeq}K_{\mathcal X^\times}$$
we have
\begin{equation}
 h_{\rho(t)_*\Omega_1^n\wedge \pi^*\omega_{FS}} = \alpha^*(p_1^* h_{\Omega_1^n} \otimes p_{2}^* h_{\omega_{FS}})\label{eq:ricmove}
\end{equation}
where $p_{2}\colon X\times \mathbb P^1\setminus \{0\}\to \mathbb P^1\setminus \{0\}$ is the projection.   But this is clear since $p_{2}\circ\alpha =\pi$ so
$$\rho(t)_*\Omega_1^n \wedge \pi^*\omega_{FS} =( \alpha^* p_1^*\Omega_1^n)\wedge \pi^*\omega_{FS} = \alpha^*(p_1^*\Omega_1^n \wedge p_{2}^*\omega_{FS}).$$
Then taking the curvature of \eqref{eq:ricmove} gives
$$ \Rel(\rho(t)_*\Omega_1) + 2\pi^*\omega_{FS} = \alpha^*p_1^* \Ric(\Omega_1) + \pi^* \Ric(\omega_{FS}) = \rho(t)_*\Ric(\Omega_1) + 2\pi^*\omega_{FS}$$
since  $\Ric(\omega_{FS}) = 2\omega_{FS}$.
\end{proof}
\color{black}

\begin{definition}[Mabuchi current]Given a smooth test-configuration $(\X,\co)$ for $(X,[\omega])$, the \emph{Mabuchi current} on $\pr^1$ is defined to be $$\eta(\co) = \pi_*\left(\frac{n}{n+1}\mu(X,[\omega])\co^{n+1}-\Rel\co\wedge\co^n\right).$$ \end{definition}

The construction of the Deligne current instead uses the usual Ricci curvature, with the addition of a relative term coming from $\pr^1$. 

\begin{definition}[Deligne current] Given a smooth test-configuration $(\mathcal X,\Omega)$ for $(X,[\omega])$ we define the \emph{Deligne current} to be $$\hat{\eta}(\co) = \pi_*\left(\frac{n}{n+1}\mu(X,[\omega])\co^{n+1}-(\Ric\co - 2\pi^*\omega_{FS})\wedge\co^n\right).$$ \end{definition}

We name this the Deligne current by analogy with the Deligne pairing (as used, for example, in \cite{PRS}).  It is defined so that
 $$\DF(\X,\co) = \int_{\pr^1} \hat{\eta}(\co).$$

\color{black}

Just as with the Aubin-Mabuchi energy, we have the following.

\begin{proposition}\label{MabuchiMetricFunctional} Let $(\mathcal X,\Omega)$ be a smooth test-configuration. Then the Mabuchi current is smooth on $\pr^1 \backslash \{0\}$. Away from $t=0$, we have $\eta(\co) = i\ddbar \scM(\varphi_t)$. \end{proposition}

\begin{proof} As the direct image of a smooth form, the Mabuchi current is smooth on $\pr^1 \backslash \{0\}$.

Since we have already considered the Aubin-Mabuchi term, we only need to consider the term $\Rel\co\wedge\co^n$. We consider the terms $$A(t)=-i\ddbar \int_X \varphi_t \left(\Ric\co_1\wedge\left(\sum_{i=0}^{n-1} \co_1^i\wedge\rho(t)^*\co_t^{n-1-i}\right)\right)$$ and $$B(t)=i\ddbar \int_X \log\left(\frac{\rho(t)^*\co_t^n}{\co_1^n}\right)\rho(t)^*\co_t^n$$ arising from the Mabuchi functional separately. We have to show $$-\int_U f\pi_*(\Rel\co\wedge\co^n) = \int_U f (A(t)+B(t))$$
where $U\subset \mathbb C^*$ is open and $f$ is a compactly supported test function on $U$. 

For the $A(t)$ term, arguing as with the Aubin-Mabuchi energy, we wish to calculate $\int_U f A(t)$ which equals 
$$-\int_{\pi^{-1}(U)}(\rho(t)_*\varphi_t)(i\ddbar (\pi^*f))\wedge \left((\rho(t)_*\Ric\co_1) \wedge\left(\sum_{i=0}^{n-1}(\rho(t)_*\co_1^i)\wedge\co^{n-1-i}\right)\right).$$ Using again the same argument as with the Aubin-Mabuchi functional, we see $$\int_U fA(t) = -\int_{\pi^{-1}(U)}(\pi^*f)((\rho(t)_*\Ric\co_1) \wedge(\co^{n}-\rho(t)_*\co_1^n)).$$ The term involving only $\co_1$ and $\Ric(\co_1)$ vanishes, leaving 

\begin{equation}
\int_U f A(t) = -\int_{\pi^{-1}(U)}(\pi^*f)((\rho(t)_*\Ric\co_1) \wedge\co^{n})\label{eq:At}
\end{equation}

For the $B(t)$ term, we have
\begin{align*}
\int_U f B(t) &=\int_U (i\ddbar f)\int_{X} \log\left(\frac{\rho(t)^*\Omega_t^n}{\Omega_1^n} \right)\Omega_t^n \\
&= \int_U (i\ddbar f)\int_{\mathcal X_t}  \log\left(\frac{\Omega_t^n}{\rho(t)_*\Omega_1^n }\right) \Omega_t^n \\
&= \int_{\pi^{-1}(U)} \log\left(\frac{\Omega^n\wedge \pi^* \omega_{FS}}{\rho(t)_*\Omega_1^n \wedge \pi^*\omega_{FS}}\right)(i\ddbar \pi^* f)\wedge \Omega^n \\
&=  \int_{\pi^{-1}(U)}\pi^* f \left(i\ddbar \log\left(\frac{\Omega^n\wedge \pi^* \omega_{FS}}{\rho(t)_*\Omega_1^n \wedge \pi^*\omega_{FS}}\right)\right)\wedge \Omega^n 
\end{align*}
Now using Lemma \ref{lem:ricmove}
\begin{align*}
-i\ddbar \log\left(\frac{\Omega^n\wedge \pi^* \omega_{FS}}{\rho(t)_*\Omega_1^n \wedge \pi^*\omega_{FS}}\right) &= \Rel(\Omega) - \Rel(\rho(t)_*\Omega_1)\\
& = \Rel(\Omega) - \rho(t)_*\Ric(\Omega_1).
\end{align*}

So combining with \eqref{eq:At} gives
$$\int_U f (A(t) + B(t)) = -\int_U f \pi_*(\Rel(\Omega)\wedge \Omega^n))$$
as required.
\color{black}

\end{proof}

While the Mabuchi current is related to the Mabuchi functional, the object we are more interested in is the Donaldson-Futaki invariant, which is related to the Deligne current. We relate the two objects as follows.

\begin{lemma}\label{MabuchivsDeligne}  Assume $(\mathcal X,\Omega)$ is a smooth test-configuration.  Then the Deligne current and the Mabuchi current are related  by $$\hat{\eta}(\co) - \eta(\co) = -i\ddbar \int_{\X_t} \log\left(\frac{\pi^*\omega_{FS}\wedge\co^n}{\co^{n+1}}\right)\co_t^n.$$ This equality is global, and in particular holds in a neighbourhood of $0\in\pr^1$.\end{lemma}

\begin{proof} 
\begin{align*}
-\Ric(\Omega) + 2\pi^*\omega_{FS} + \Rel(\Omega) &= i\ddbar\log \Omega^{n+1} -i\ddbar\log(\Omega^n\wedge \pi^*\omega_{FS})\\
&= i\ddbar\log \left(\frac{\Omega^{n+1}}{\Omega^n\wedge \pi^*\omega_{FS}}\right)
\end{align*}
Thus
$$ \hat{\eta}(\co) - \eta(\co) = \pi_*\left(i\ddbar\log\left( \frac{\co^{n+1}}{\pi^*\omega_{FS}\wedge\co^n}\right)\wedge\co^n\right).$$
One readily sees that $$\int_U f i\ddbar \int_{\X_t} \log\left(\frac{\co^{n+1}}{\pi^*\omega_{FS}\wedge\co^n}\right)\co_t^n = \int_U f \pi_*\left(\left(i\ddbar\log \frac{\co^{n+1}}{\pi^*\omega_{FS}\wedge\co^n}\right)\wedge\co^n\right)$$  where $U\subset\pr^1$ is an arbitrary open subset.
\end{proof}

\begin{remark}\label{psi-considerations} The previous lemma is closely analogous to a result is proven in \cite{PRS}. As noted in \cite{PRS}, the function $$f_t = \frac{\pi^*\omega_{FS}\wedge\co^n}{\co^{n+1}}$$ is the ratio of two top degree forms whose denominator is strictly positive, and so $f_t$ is a non-negative smooth function on $\X_t$  (it is important here that $\X$ is smooth, so that $\co$ is K\"ahler; this would not be true on a resolution of singularities $p:\Y\to\X$ where $p^*\co$ is merely semi-positive). It follows that the integral $$\psi(t)= \int_{\X_t} \log\left(\frac{\pi^*\omega_{FS}\wedge\co^n}{\co^{n+1}}\right)\co_t^n$$ is smooth and bounded above as $t\to 0$. Moreover, as in \cite{PRS}, it is also continuous and  bounded below provided the central fibre $\X_0$ is reduced. With this notation, the previous lemma proves that $$\hat{\eta}(\co) - \eta(\co) = -i\ddbar \psi(t).$$
 \end{remark}

Finally, we consider the minimum norm. We assume $(\X,\co)$ admits a map $q: \X\to X\times\pr^1$, working on a resolution of indeterminacy if not. As in Remark \ref{minimum-norm-using-co_1}, we fix an isomorphism $X\cong\X_1$, and consider $\co_1$ to be a K\"ahler metric on $X$. We also denote $$\scN(\co) = \pi_*\left ((\co\wedge q^*(\co_1)^n) - \frac{\co^{n+1}}{n+1}\right).$$

\begin{proposition}\label{J-functional-prelim} Away from $t=0$, we have $\scN(\co)= i\ddbar J(\varphi_t)$.\end{proposition}

\begin{proof}
The metric $q^*\co_1$ satisfies the property that $(q^*\co_1)|_{\X_t} = \rho(t)_*\co_1$. Here $q^*\co_1$ denotes the $(1,1)$ form on $\X$, and $(q^*\co_1)|_{\X_t}$ is its restriction to $\X_t$.

Since we have already considered the Aubin-Mabuchi energy, it suffices to show $$\pi_*(\co\wedge \rho(t)_*\co_1) = i\ddbar \left(\int_X\varphi_t\co_1^n\right).$$ The proof then follows exactly as in the previous cases.

\end{proof}

\subsection{A Kempf-Ness lemma}

We now relate the numerical invariants, such as the Donaldson-Futaki invariant, to the limit derivatives of the corresponding functionals along certain paths of metrics. The simplest case, which we deal with first, is the Aubin-Mabuchi functional. So far we have used $t$ as the coordinate; to state our results it will be more convenient to use $\tau = -\log |t|^2$. In this way, $\tau\to\infty$ corresponds to $t\to 0$, and $t=e^{-\tau/2}$. It will also be useful to use the notation $\theta_{\tau} = \varphi_{t}$.

\begin{theorem}\label{VolumeIsAM} Let $(\X,\co)$ be a test-configuration. Then $$[\co]^{n+1} = \lim_{\tau \to \infty}\frac{d}{d\tau} \AM(\theta_{\tau}).$$\end{theorem}

\begin{proof}

Using Proposition \ref{AMasddbar}, we have \begin{align*}[\co]^{n+1} &= \int_{\pr^1} \pi_*(\co^{n+1}), \\ &= \int_{\C = \pr^1 \backslash \{0\}} \pi_*(\co^{n+1}), \\ &= \int_{\C} i\ddbar \AM(\varphi_{t}). \end{align*}

The function $\AM(\varphi)$ is $S^1$-invariant, since $\co$ is. For clarity we write $\nu(t) = \AM(\varphi_{t})$. The subsequent argument follows \cite[Lemma 2.6]{bermankpoly}: with these coordinates, we have \begin{align*}\int_{\C} i\ddbar \AM(\varphi_{t}) &= \int_{-\infty}^{\infty} d \left(\frac{d\nu(e^{-\tau/ 2})}{d\tau}\right), \\ &= \left(\lim_{\tau\to -\infty}\frac{d}{d\tau} \nu(e^{-\tau/ 2})\right) + \left(\lim_{\tau\to\infty}\frac{d}{d\tau} \nu(e^{-\tau/ 2})\right).
\end{align*}To conclude we note that $$\lim_{\tau \to -\infty}\frac{d}{d\tau} \AM(\theta_{\tau}) = 0,$$ since the $\C^*$-action on $\X$ is trivial at infinity and hence $\theta_{\tau}$ and so  $\AM(\theta_{\tau})$ tend to a constant.

\end{proof}

We next relate the Donaldson-Futaki invariant to the Mabuchi functional. For this argument, we require that $\X_0$ is reduced, which by Remark \ref{psi-considerations} will give detailed information about the function 
$$\psi(t)= \int_{\X_t} \log\left(\frac{\pi^*\omega_{FS}\wedge\co^n}{\co^{n+1}}\right)\co_t^n.$$ 

\begin{theorem}\label{DFisslopeofMabuchi}Let $(\X,\co)$ be a smooth test-configuration with reduced central fibre. Then $$\DF(\X,\co) = \lim_{\tau\to\infty} \frac{d}{d\tau} \scM(\theta_{\tau}).$$
\end{theorem}

\begin{proof}

By definition of the Deligne current we have \begin{align*}\DF(\X,\co)&= \int_{\pr^1}\hat{\eta}(\co), \\ &= 
\int_{\pr^1}(\eta(\co)+ (\hat{\eta}(\co) - \eta(\co))).\end{align*}

We consider each term separately, dealing first with the $\eta(\co)$ term. By smoothness, the mass given by the Deligne current to $\{0\}$ is zero, just as with $\pi_*\co^{n+1}$ in Proposition \ref{AMasddbar}. Using Theorem \ref{MabuchiMetricFunctional}, we have \begin{align*}\int_{\pr^1}\eta(\co) &= \int_{\C = \pr^1\backslash \{0\}}\eta(\co), \\ &= \int_{\C} i\ddbar \scM(\varphi_t).\end{align*} 

Denoting temporarily $\nu(t) = \scM(\varphi_{t})$, using the same argument as \cite[Lemma 2.6]{bermankpoly} again, we have \begin{align*}\int_{\C}i\ddbar (\scM(\varphi_{t})) &= \int_{-\infty}^{\infty} d\left(\frac{d\nu(e^{-\tau/2})}{d\tau}\right), \\ &= \lim_{\tau \to\infty} \frac{d\nu(\tau)}{d\tau}, \\ &=  \lim_{\tau \to\infty} \frac{d\scM(\theta_{\tau})}{d\tau}.\end{align*}  Here the limit as $\tau\to -\infty$ vanishes for the same reason as in Theorem \ref{VolumeIsAM}.

We now calculate $$\int_{\pr^1}(\hat{\eta}(\co) - \eta(\co)) = -\int_{\pr^1}i\ddbar\psi(t).$$ Since $\hat{\eta}(\co)$ and $\eta(\co)$ are well defined as currents, so is $i\ddbar\psi(t)$. As $\X_0$ is reduced, by Remark \ref{psi-considerations} $\psi(t)$ is continuous, hence bounded, and smooth away from $t=0$. It follows that \begin{align*}\int_{\pr^1}i\ddbar\psi(t) &= \int_{\pr^1\backslash \{0\}}i\ddbar\psi(t), \\ &= -\lim_{\tau \to\infty} \frac{d\psi(e^{-\tau/2})}{d\tau}.\end{align*} This derivative vanishes as  $\psi(t)$ is bounded, completing the proof.

\end{proof}

\begin{corollary}\label{cor:K-semistability} Suppose the Mabuchi functional for $[\omega]$ is bounded.  Then $(X,\omega)$ is K-semistable.  In particular this holds if $[\omega]$ admits a cscK metric. \end{corollary} 

\begin{proof} By Proposition \ref{prop:simplifiedtestconfig} we need only consider test-configurations for $(X,[\omega])$ with smooth total space and reduced central fibre.  Then boundedness of the Mabuchi functional implies
$$\lim_{\tau\to\infty}   \frac{d\scM(\theta_{\tau})}{d\tau} \geq 0$$
which from the above gives K-semistability.  The second statement follows from Berman-Berndtssson \cite{BB} who prove the existence of a cscK metric in $[\omega]$ implies the Mabuchi functional is bounded.
\end{proof}

We move on to consider the case the Mabuchi functional is coercive, for which we require a similar result regarding the J-functional. We remark that the following result applies to general, singular test-configurations.

\begin{theorem}\label{MinNormAsSlope} Let $(\X,\co)$ be a test-configuration. Then $$\|(\X,\scL)\|_m = \lim_{\tau\to\infty} \frac{d}{d\tau} J(\theta_{\tau}).$$ 
\end{theorem}

\begin{proof} 

The proof is the same as for the Aubin-Mabuchi functional, using Proposition \ref{J-functional-prelim}.  For this we use that the direct image current $\pi_*\left ((\co\wedge q^*(\co_1)^n) - \frac{\co^{n+1}}{n+1}\right)$ gives zero mass to $\{0\}$ as the mass equals the integral $$\int_{\X_0} \left((\co\wedge q^*(\co_1)^n) - \frac{\co^{n+1}}{n+1}\right) = 0,$$ using the same notation as Proposition \ref{J-functional-prelim}. If the test-configuration is singular, we pass to a resolution of singularities and run the same argument.
\end{proof}

The following is then an immediate corollary.

\begin{corollary} \label{cor:Kstable} Suppose $(X,[\omega])$ has coercive Mabuchi functional. Then $(X,[\omega])$ is uniformly K-stable.    In particular this holds if the automorphism group of $(X,[\omega])$ is discrete, and $[\omega]$ admits a K\"ahler metric that has constant scalar curvature.   \end{corollary}

\begin{proof} In this case the functional $\scM(\varphi) - \epsilon J(\varphi)$ is bounded below, so in particular its limit derivative along a sequence of K\"ahler potentials must be non-negative. But, for an arbitrary test-configuration (assuming the total space is smooth and the central fibre is reduced, which we may do by Proposition \ref{prop:simplifiedtestconfig}), we have \begin{align*}\DF(\X,\co) - \epsilon \|(\X,\co)\|_m &= \lim_{\tau\to\infty} \frac{d}{d\tau} (\scM(\theta_{\tau}) - \epsilon J(\theta_{\tau})), \\ &\geq 0.\end{align*} It follows that $\DF(\X,\co) \geq \epsilon \|(\X,\co)\|_m,$ i.e. $(X,[\omega])$ is uniformly K-stable.  The second statement then follows from the main result of Berman-Darvas-Lu \cite{BermanDarvasLu}.
\end{proof}

Another corollary is the following, which proves Proposition \ref{MinimumNormIsNonnegative}. 

\begin{proposition}\label{min-norm-is-nonneg-proof} The minimum norm of a test-configuration is non-negative. \end{proposition}

\begin{proof} By Lemma \ref{IJarenonneg}, the $J$-functional is \emph{always} bounded below by zero. In particular, one has $\lim_{\tau\to\infty} \frac{d}{d\tau} J(\theta_{\tau})\geq 0$. It then follows from Theorem \ref{MinNormAsSlope} that $\|(\X,\scL)\|_m\geq 0$.\end{proof}

\section{Stoppa's Theorem for K\"ahler manifolds}\label{sec:Stoppa}

The goal of the present section is to prove the following.

\begin{theorem}\label{bodystrictstability} If a K\"ahler manifold with discrete automorphism group admits a cscK metric, then it is K-stable. \end{theorem}

Our argument is similar in spirit to Stoppa's method in the projective case \cite{Stoppablowup}. As with Stoppa's method, the method relies on the following result due to Arezzo-Pacard \cite{ArezzoPacard}.   For a point $p\in X$ we let $Bl_pX$ denote the blowup of $X$ at $p$ with exceptional divisor $F$.  

\begin{theorem}\label{Arezzo-Pacard}  Suppose $(X,[\omega])$ is a K\"ahler manifold with discrete automorphism group admitting a cscK metric, and let $p\in X$. Then for all $\epsilon$ sufficiently small, $(Bl_pX,[\omega]-\epsilon c_1(F))$ admits a cscK metric. \end{theorem}

The argument is then as follows. We wish to show that the Donaldson-Futaki invariant of each non-trivial test-configuration $(\mathcal X,\co)$ for $(X,[\omega])$ is strictly positive.  By the K-semistability result from the previous section we have only to consider the possibility of such a test-configuration with zero Donaldson-Futaki invariant.   Let $p\in X$ and denote by $C$ the closure of the $\mathbb C^*$ orbit of the point $p\in X=\X_1$ inside $\X$.  Assuming $\X$ and $C$ are smooth, then by blowing up along $C$, with exceptional divisor $E$, we get for $\epsilon$ sufficiently small a test-configuration $(\Y,\co-\epsilon \xi_\epsilon)$ for $(Bl_pX,[\omega]-\epsilon c_1(F))$  where $\xi_\epsilon\in c_1(E)$ is some smooth $(1,1)$-form in $c_1(E)$. The goal is to explicitly calculate the change of the Donaldson-Futaki invariant as $\epsilon$ tends to zero, and show that this contradicts the K-semistability result for the blowup.

We start by giving the following definition by analogy with Stoppa's work:

\begin{definition}[Chow weight] Given a test-configuration $(\X,\co)$ for $(X,[\omega])$ and a point $p\in X$ the \emph{Chow weight} is defined to be $$Ch_p(\X,\co) = \frac{[\co]^{n+1}}{(n+1)[\omega]^n} - \int_{C}\co$$
where $C$ is the closure of the $\mathbb C^*$-orbit containing $p$. \end{definition}

With this definition in place, we will prove the following.

\begin{proposition}\label{inducedDF}Suppose $\X$ and $C$ are smooth. The Donaldson-Futaki invariant of $(\Y,\co-\epsilon\xi_{\epsilon})$ is 
$$\DF(\Y,\co-\epsilon\xi_{\epsilon}) = \DF(\X,\co) - n(n-1)\epsilon^{n-1}Ch_p(\X,\co)+O(\epsilon^n).$$\end{proposition}

If $\X$ or $C$ is singular,  we will produce in Section \ref{KahlerStoppaFormula} a test configuration $(\B,\co_{\B})$ for  $(Bl_pX,[\omega]-\epsilon c_1(F))$ using a slightly different construction which similarly satisfies $$\DF(\B,\co_{\B}) = \DF(\X,\co) - n(n-1)\epsilon^{n-1}Ch_p(\X,\co)+ O(\epsilon^n).$$

Thus the proof of Theorem \ref{bodystrictstability} is completed by the following result:

\begin{proposition}\label{point-choice} For any (possibly singular) test-configuration $(\X,\co)$ that satisfies $\|(\X,\co)\|_m>0$, there exists a point $p\in X$ such that $$Ch_p(\X,\co)>0.$$ \end{proposition}

\begin{proof}[Proof of Theorem \ref{bodystrictstability}] Suppose  $(X,[\omega])$ is K-semistable but not K-stable. Then there is a test-configuration $(\X,\co)$ for $(X,[\omega])$ with $\|(\X,\co)\|_m>0$ but $\DF(\X,\co)=0$. Choosing $p$ as in Proposition \ref{point-choice}, for $\epsilon$ sufficiently small, we have $\DF(\Y,\co-\epsilon \xi_{\epsilon})<0$ (or $\DF(\B,\co_{\B})<0$ if $\X$ or $C$ is singular). Therefore $(Bl_pX,[\omega]-\epsilon c_1(F))$ is K-unstable. This contradicts the K-semistability, which follows from the combination of Theorem \ref{Arezzo-Pacard} and Corollary \ref{cor:K-semistability}. Therefore $(X,[\omega])$ must be K-stable as claimed.
\end{proof}

\subsection{Discussion of the projective case}\label{sec:stoppaproj}
Before proving the required results for K\"ahler manifolds, we briefly discuss the projective case.   We will assume here that various quantities involved are smooth to make the exposition clearer (in any case we will prove the corresponding results in the more general K\"ahler setting in Section \ref{KahlerStoppaFormula}). So let $(X,L)$ be a smooth projective variety and an ample $\R$-line bundle. Let $(\X,\scL)$ be a smooth test-configuration for $(X,L)$, with $\scL$ an $\R$-line bundle. 

Fix a point $p\in X$, and denote $C$ the closure $\overline{\C^*.p}$. We assume that $C$ is smooth, so that $C\cong\pr^1$. Set $\psi:\B=Bl_C\X \to \X$, with exceptional divisor $E$. Then $(\B,\scL-\epsilon E)$ is clearly a test-configuration for $(Bl_pX,L-\epsilon F)$.  We wish to calculate the Donaldson-Futaki invariant of  $(\B,\scL-\epsilon E)$. We have
$$K_{\B} = \psi^*K_{\X}+(n-1)E,$$ and also

\begin{align*}(L-\epsilon F)^n &= L^n -\epsilon^n, \\ -K_{\Bl_pX}.(L-\epsilon F)^{n-1} &= -L^{n-1}.K_X - (n-1)\epsilon^{n-1}.\end{align*} 

Similarly, by the projection formula, we see \begin{align*}(\scL-\epsilon E)^{n+1} &= \scL^{n+1} + (n+1)(-\epsilon E)^n.\scL + O(\epsilon^{n+1}), \\ K_{\B/\pr^1}.(\scL-\epsilon E)^n &= \scL^n.K_{\X/\pr^1}+n(n-1)(-\epsilon E)^{n-1}.E.\scL + O(\epsilon^n).\end{align*}

The following standard intersection-theoretic result relates the computed intersection numbers to intersections on $\X$.

\begin{lemma}\label{projective-intersections} $(-E)^n.\scL = \scL.C.$\end{lemma}

Defining  $$Ch_p(\X,\scL)=\frac{\scL^{n+1}}{(n+1)L^n} - \scL.C$$ and putting this together gives that the Donaldson-Futaki invariant of $(\Y,\scL-\epsilon E)$ is
\begin{equation}
\DF(\B,\scL-\epsilon E) = \DF(\X,\scL) - n(n-1)\epsilon^{n-1}Ch_p(\X,\scL) + O(\epsilon^n) \label{DF-blowup-projective} 
\end{equation}
which is nothing other than Proposition \ref{inducedDF} in the projective case \cite{Stoppablowup,StoppaUnstable} (the point we are emphasising here once again is that such calculations are made easier if one considers the Donaldson-Futaki invariant in terms of intersection theory on the total space of the test-configuration).  Thus one is left to prove the following statement:

\begin{theorem}(Stoppa)\label{thm:stoppaChow} Suppose $(\X,\scL)$ is not the trivial test-configuration $(X\times\pr^1,L)$. Then there exists a point $p\in X$ such that $$Ch_p(\X,\scL)=\frac{\scL^{n+1}}{(n+1)L^n} - \scL.C >0.$$\end{theorem}

This statement is proved by Stoppa, and we shall not repeat his argument here.  However given the point of view we have been taking, it has an equivalent and simple formulation.  We may normalise so $\mathcal L^{n+1}=0$.  Suppose for contradiction that no $p$ exists.  Then $\mathcal L.C\geq 0$ for all such curves $C$.  But $\mathcal L$ is relatively ample so this implies $\mathcal L.C\ge 0$ for all $\mathbb C^*$-invariant curves $C$ inside $\X$ and so $\mathcal L$ is nef.  Thus Theorem \ref{thm:stoppaChow} is equivalent to the following statement:

\begin{theorem} Suppose $(\X,\scL)$ is a test-configuration such that $\scL$ is nef but not big. Then $(\X,\scL)$ is the trivial test-configuration, i.e. $(\X,\scL) \cong (X\times\pr^1, L)$.\end{theorem}

It should be possible to give a direct proof of this using basic positivity properties of line bundles (and thus completely avoid the Geometric Invariant Theory arguments of Stoppa), but we will not consider that question further here.

\begin{remark} The results extend easily to the higher dimensional case (compare Della Vedova \cite{DellaVedova}). Fix a smooth subvariety $Z\subset X$ of dimension $m$,  so that the case we have considered is $\dim Z =0$. Let $(\X,\scL)$ be a smooth test-configuration, and suppose the closure $\scZ$ of the $\C^*$-orbit of $Z$ in $\X$ is smooth. Then the blow-up $\B=\Bl_{\scZ}\X$ is a test-configuration for $Bl_Z X$, just as above. Arguing as in the case $m=0$, its Donaldson-Futaki invariant admits an expansion of the form $$ \DF(\B,\scL-\epsilon E) = \DF(\X,\scL) + c\epsilon^{n-m-1}\DF_Z(\X,\scL) + O(\epsilon^{n-m}),$$ where $c>0$ and $\DF_Z(\X,\scL)$ is the leading term in the expansion of the $X$-twisted asymptotic Chow weight of $(Z,L|_Z)$, introduced by the first author and Keller \cite{DervanKeller}. The smoothness hypotheses are satisfied, for example, when $(\X,\scL)$ is a test-configuration induced by a holomorphic vector field on $(X,L)$. This answers a question of Stoppa \cite[Remark 4.13]{StoppaUnstable}. Similarly, using this technique it is straightforward to compute the lower order terms in the expansion of the Donaldson-Futaki invariant given in equation (\ref{DF-blowup-projective}).\end{remark} 

\subsection{The Donaldson-Futaki invariant of the induced test-configuration}\label{KahlerStoppaFormula}

We return to the K\"ahler setting.  Fix a test-configuration $(\X,\co)$ for $(X,[\omega])$, and fix a point $p\in X$. Denote by $C=\overline{\C^*.p}$, the closure of the $\C^*$-orbit of $p$. From this data, we construct a test-configuration for the blow-up of $X$ at $p$, and compute its Donaldson-Futaki invariant. For now assume both $\X$ and $C$ are smooth, so $C$ is automatically isomorphic to $\mathbb P^1$.

Let $F$ be the exceptional divisor of the blow-up $\nu_1: B\to X$ at a point $p$, and let $\zeta\in c_1(F)$. Then $\nu_1^*[\omega]-\epsilon c_1(F)$ is a K\"ahler class for $\epsilon$ sufficiently small.

Similarly, let $\nu_2: \B = Bl_C\X\to\X$ be the blowup of $\X$ along $C$, so $\nu_2^*[\Omega]-\epsilon[E]$ is a K\"ahler class for $\epsilon$ sufficiently small.  Thus there exists a family of smooth $(1,1)$-forms $\xi_{\epsilon}\in c_1(E)$, such that $\co-\epsilon \xi_\epsilon$ is K\"ahler on $\Y$ for $\epsilon$ sufficiently small (here we have abused notation by writing $\co$ as the pullback of $\co$ to $\B$).

\begin{lemma} $\B$ admits a $\C^*$-action, and $\xi_{\epsilon}$ can be chosen $S^1$-invariant. As such, $(\B,\co-\epsilon\xi_{\epsilon})$ is a test-configuration for $(B, [\omega]-\epsilon c_1(F))$.\end{lemma}

\begin{proof} Since $C$ is $\C^*$-invariant, $\B$ automatically admits a $\C^*$-action. If $\xi$ is not $S^1$-invariant, it can be chosen to be so by averaging.\end{proof}

We now compute the Donaldson-Futaki invariant of this induced test-confi\-guration.

\begin{lemma} The volume term in the Donaldson-Futaki invariant changes as $$[\co-\epsilon\xi_{\epsilon}]^{n+1} = [\co]^{n+1}+(-\epsilon)^n(n+1)[\xi_{\epsilon}]^n.[\co] + O(\epsilon^{n+1}).$$ \end{lemma}

\begin{proof} This is obvious using the projection formula, since $C$ is one-dimensional.\end{proof}

 \begin{lemma}\label{c_1-of-blowup}The remaining term in the Donaldson-Futaki invariant changes as $$[\co-\epsilon\xi_{\epsilon}]^n.c_1(\B) = [c_1(\X)].[\co]^n + (-\epsilon)^{n-1}n(n-1)[\xi_{\epsilon}]^{n}.[\co] + O(\epsilon^n).$$ Here $[c_1(\X)].[\co]^n$ is calculated on $\X$ and the remaining integrals are calculated on $\B$.\end{lemma}

\begin{proof}
Remark that $[\xi_{\epsilon}]\in c_1(E)$. As $\B$ is the blow-up of $\X$ along the one-dimensional submanifold $C$, we have in cohomology $c_1(\B) = \nu_2^*c_1(\X) - (n-1)[E]$. The result then follows from the projection formula.
 \end{proof}
 
The final step in relating the Donaldson-Futaki invariant of $(\B,\co-\epsilon\xi)$ to integrals calculated on $\X$ is the following.
 
 \begin{lemma}\label{KahlerIntersectionNumberForms} $-[-\xi]^n.[\co] = \int_C\co.$\end{lemma}
 
 \begin{proof} $\co|_C$ is in $\Amp_{\R}(C)$, since $C$ is one-dimensional and hence every K\"ahler class is a limit of classes of ample $\Q$-line bundles. Fix an $\R$-line bundle $H$ on $C$ such that $\co|_C\in c_1(H)$. Let $\alpha: E\to C$ be the map induced from the blow-up. Remark that $E$ is projective, as it is the projective bundle over $C$. Moreover $E$ gives a line bundle on $\B$, hence by restriction a line bundle on $E$ which we denote (abusing notation) by $E$. Then what we wish to prove is $$(-E|_E)^{n-1}.(\alpha^*H) = \deg H,$$ where the intersection number on the left hand side is computed on $E$ and $\deg H$ is computed on $C$. This is a standard argument using intersection theory for projective varieties when $H$ is a $\Z$ or $\Q$-line bundle as in Lemma \ref{projective-intersections}, and therefore extends to $\R$-line bundles by continuity.\end{proof}

Combining the previous lemmas gives the following.
 
\begin{proposition}\label{KahlerBlowupDF} The Donaldson-Futaki invariant of $(\Y,\co-\epsilon \xi_{\epsilon})$ satisfies $$\DF(\Y,\co-\epsilon \xi_{\epsilon}) = \DF(\X,\co) - n(n-1)\epsilon^{n-1}Ch_p(\X,\co)+ O(\epsilon^n).$$\end{proposition}

\begin{proof} All that remains is to compute the slope of $(\B,[\omega-\epsilon \zeta])$ in terms of the slope of $(X,[\omega])$. This follows easily as above.\end{proof}

Now we relax the smoothness assumption, so let $(\X,\co)$ be an arbitary test-configuration and $p\in X$. Denote by $C$ the closure $\overline{\C^*.p}$.   Take a resolution of singularities $p: \Y\to\X$ such that the proper transform $\hat C$ of $C$ is smooth. Then $p^*\co$ is semi-positive, and K\"ahler away from $\X_0$. Set $b: \B = Bl_{\hat C} \Y\to \Y$ to be the blow-up, so that we have a diagram as follows.

\[
\begin{tikzcd}
\B \arrow[swap]{d}{b} \arrow{dr}{q=p \circ b} &  \\
\Y \arrow{r}{p} & \X
\end{tikzcd}
\]

Remark that $\B$ is a smooth K\"ahler manifold. Let $E$ be the exceptional divisor of $\Y\to\X$ and $\hat E$ the exceptional divisor of $\B\to\Y$. Take $\epsilon>0$ and $\xi_{\epsilon}\in c_1(\hat E)$ and $\alpha_{\epsilon} \in c_1(E)$ such that $q^*\co-\epsilon\xi_{\epsilon}-\epsilon^{n} b^*\alpha_{\epsilon}$ is K\"ahler on $\B$. We can, and do, assume that $q^*\co-\epsilon\xi_{\epsilon}-\epsilon^{n} b^*\alpha_{\epsilon}$ is $S^1$-invariant. Remark that $E$ has support in the central fibre of the test-configuration, so in cohomology we have $$[(q^*\co-\epsilon\xi_{\epsilon}-\epsilon^{n} b^*\alpha_{\epsilon})_t]\cong [\omega]-\epsilon c_1(F).$$ Then $(\B,q^*\co-\epsilon\xi_{\epsilon}-\epsilon^{n} b^*\alpha_{\epsilon})$ is a test-configuration for $(Bl_pX,[\omega]-\epsilon c_1(F))$. The following is immediate.

\begin{lemma} We have $\DF(\B,q^*\co - \epsilon\xi_{\epsilon}-\epsilon^{n} b^*\alpha_{\epsilon})\geq 0$. Moreover, setting $\mu = \mu(Bl_pX,[\omega]-\epsilon c_1(F))$, the Donaldson-Futaki invariant $$\DF:=\DF(\B,q^*\co - \epsilon \xi_{\epsilon}-\epsilon^{n} b^*\alpha_{\epsilon})$$ is given as 
\begin{align*} \DF= \frac{n}{n+1}\mu[q^*\co-\epsilon\xi_{\epsilon}]^{n+1} - [c_1(\B) - (\pi\circ p)^*c_1(\pr^1)].[q^*\co-\epsilon\xi_{\epsilon}]^n + O(\epsilon^{n}).\end{align*}
 \end{lemma}

Just as with the minimum norm proven in Proposition \ref{canworkonres}, one can calculate the Chow weight $Ch_p(\X,\co)$ on a resolution of singularities.

\begin{lemma} We have $$Ch_p(\X,\co) = \frac{[p^*\co]^{n+1}}{(n+1)[\omega]^n} - \int_{\hat C}p^*\co.$$\end{lemma}

\begin{proof} It suffices to show $$\int_C \co =  \int_{\hat C}p^*\co,$$ where $\hat C$ is the proper transform of $C$ under the map $p:\Y\to\X$. Denote $q = C\cap\X_0$. Then since $\co$ is smooth, we have \begin{align*}\int_C \co &= \int_{C\setminus \{q\}} \co, \\ &= \int_{\hat C\setminus \{p^{-1}(q)\}}\co, \\ & =\int_{\hat C}p^*\co, \end{align*} where we have used that $p: \hat C\setminus \{p^{-1}(q)\} \cong C\setminus \{q\}$ is an isomorphism away from $p^{-1}(q)$.
\end{proof}

Remark that $\B\to\Y$ is a map of smooth K\"ahler manifolds, which is the blow-up along the smooth curve $\hat C$. The Donaldson-Futaki invariant of $(\X,\co)$ is computed on $\Y$, by definition. Thus Proposition \ref{KahlerBlowupDF} gives the following.

\begin{proposition}\label{singularKahlerBlowupDF} The Donaldson-Futaki invariant of $(\B,q^*\co-\epsilon \xi_{\epsilon}-\epsilon^{n} \alpha_{\epsilon})$ satisfies $$\DF(\B,q^*\co-\epsilon \xi_{\epsilon}-\epsilon^{n} \alpha_{\epsilon}) = \DF(\X,\co) - n(n-1)\epsilon^{n-1}Ch_p(\X,\co)+ O(\epsilon^n).$$\end{proposition}

\subsection{Existence of a destabilising point}\label{existence-of-point-section}
 
As in the last section, we fix a test-configuration $(\X,\co)$ for $(X,[\omega])$. The goal of this section is to prove the following.

\begin{proposition}\label{prop-point} Provided $\|(\X,\co)\|_m>0$, there exists a point $p\in X$ satisfying $$Ch_p(\X,\co)>0.$$\end{proposition} 

\begin{remark} Due to pathological examples of Li-Xu \cite[Section 8.2]{LiXu}, we require the assumption that $\|(\X,\co)\|_m>0$ (rather than the \emph{a priori} weaker assumption that $(\X,\co)$ is not equivariantly isomorphic to $(X\times\pr^1,p_X^*\omega)$ for some $\omega$). We refer to \cite{JS-note} for a discussion of this requirement.\end{remark}

To prove this it will be useful to normalise such that $[\co]^{n+1}=0$, using Lemma \ref{normalisingtheDF} (that is we replace $\co$ by $\co + \pi^*\omega_{FS}$ which does not change the Donaldson-Futaki invariant or minimum norm, but it should be noted that $\co$ will no longer be K\"ahler). Recall by Theorem \ref{VolumeIsAM} that we have $$\lim_{\tau \to \infty} \frac{d}{d\tau}\AM(\theta_{\tau}) = [\co]^{n+1},$$ where as usual $\rho(t)^*\co_t =\co_1 + i \ddbar \varphi_t$ with $\varphi_t$ are normalised as in Definition \ref{def:normalisationofpotentials}, $\tau = -\log |t|^2$ and $\theta_{\tau} = \varphi_t$. The following is a simple calculation.

\begin{lemma}\label{AubinMabuchiNormalisation}The derivative of the Aubin-Mabuchi functional is $$\frac{d}{d\tau}\AM(\theta_{\tau}) = \frac{1}{n+1}\int_X \dot\theta_{\tau}\omega_{\theta_{\tau}}^n.$$Thus if  $[\co]^{n+1}=0$, then $$\lim_{\tau \to \infty} \int_X \dot\theta_{\tau}\omega_{\theta_{\tau}}^n =0.$$\end{lemma}

We wish to give another way of calculating the integral $\int_{C}\co$. Since $C\cong\pr^1$, and since $\co$ is smooth, we have $$\int_{C\cong \pr^1} \co = \int_{\C^*} \co.$$ 

From Theorem \ref{VolumeIsAM} we obtain the following.  \begin{lemma}\label{0-dimensionalKempfNess}We have $$\lim_{\tau\to\infty} \frac{d}{d\tau}\theta_{\tau}(p) = \int_{C} \co.$$ \end{lemma} Here we think of $(C,\co|_{C})$ as a test configuration for the point $p$, so that the Aubin-Mabuchi functional is simply $\theta_{\tau}(p)$ (the proof of Theorem \ref{VolumeIsAM} works even when $X$ is zero dimensional).

\begin{remark}
Lemma \ref{AubinMabuchiNormalisation} can also be obtained from \cite[Lemma 2.6]{bermankpoly} (noting that no positivity assumption is required in \cite[Lemma 2.6]{bermankpoly} as we assume $\co$ is smooth, and the equality we seek is additive in the K\"ahler metric).
\end{remark}

\begin{remark} With this in place, it is clear why one should expect a destabilising point to exist. Under our normalisationof $\co$, we have $$\lim_{\tau \to \infty}\int_X \dot\theta_{\tau}\omega_{\theta_{\tau}}^n=0.$$ On the other hand, if no destabilising point exists, for all $p\in X$ we have $$\lim_{\tau\to\infty} \frac{d}{d\tau}\theta_{\tau}(p) \geq 0,$$ which should lead to a contradiction. We now give a precise argument for why this is the case. \end{remark}

We would like to define the $L^1$-norm of a test-configuration, analogous to that in the projective case. It is not \emph{a priori} clear that our definition is actually well defined, so we give a conditional definition.

\begin{definition}[$L^1$-norm] Fix a test-configuration $(\X,\co)$, and normalise such that $[\co]^{n+1}=0$. We define the $L^1$\emph{-norm} of $(X,\co)$ to be $$\|(\X,\co)\|_1 = \lim_{\tau \to \infty} \int_X |\dot\theta_{\tau}| \omega_{\theta_{\tau}}^n,$$ provided the limit exists. \end{definition} 

We expect the $L^1$-norm always exists, and we will show it does in our case of interest. Our aim is to relate the $L^1$-norm to the minimum norm. The $L^1$-norm is closely related to the geometry of the space of K\"ahler potentials, as we now recall from \cite{Dar}.

\begin{definition}[$L^1$-Mabuchi metric]\label{L1Mabuchi} Denote by $\scH_{\omega}$ the space of K\"ahler potentials for $\omega$. We define the $L^1$\emph{-Mabuchi metric} on $\scH_{\omega}$ to be $$\|\xi\|_{\varphi} = \int_X |\xi|\omega_{\varphi}^n,$$ where $\xi \in T_{\varphi}\scH_{\omega}\cong C^{\infty}(X,\R)$. \end{definition}

\begin{remark} The $L^1$\emph{-Mabuchi metric} is called the \emph{weak Finsler metric} by Darvas-Rubinstein \cite[p. 10]{Dar-Rub}.\end{remark}

From this, one obtains a Finsler metric on $\scH_{\omega}$. 

\begin{definition}\label{FinslerMetric} Fix a smooth path $\alpha(s): [0,1] \to \scH_{\omega}$ with $\alpha(0)=\varphi_0$ and $\alpha(1) = \varphi_1$. The \emph{length} of $\alpha$ is defined to be $$\ell_1(\alpha) = \int_0^1 \| \dot\alpha(s) \|_{\alpha(s)} ds.$$ We define the \emph{Finsler metric} on $\scH$ as $$d_1(\varphi_0,\varphi_1) = \inf_{\alpha}\{\ell_1(\alpha) \ | \ \alpha(0)=\varphi_0, \alpha(1) = \varphi_1\}.$$\end{definition}

Our application of these ideas is related to the path of metrics arising from a test-configuration $(\X,\co)$.  We define 
$$\beta(s) = \theta_s$$
and for fixed $\tau$ let
$$ \gamma_\tau(s) = \beta(\tau s) \text{ for } s\in [0,1]. $$
In this way  $\gamma_{\tau}$ is a smooth path between $0$ and $\theta_\tau$  so
$$d_1(0,\theta_{\tau}) \leq \ell_1(\gamma_{\tau}).$$

Some elementary calculus gives the following.

\begin{lemma}\label{derivative-of-length} $\frac{d}{d\tau} \ell_1(\gamma_\tau) = \|\dot\theta_{\tau}\|_{\theta_{\tau}}$.\end{lemma}

\begin{proof} 

Clearly
$$\|\dot{\gamma}_\tau(s)\|_{\gamma_\tau(s)} = \tau\|\dot\beta(\tau s)\|_{\beta(\tau s)},$$ so 
$$\ell_1(\gamma_\tau)= \int_0^1\|\dot{\gamma_\tau}(s)\|_{\gamma_\tau(s)} ds =  \int_0^1\tau\|\dot\beta({\tau}s)\|_{\beta({\tau}s)}ds = \int_0^{\tau} \|\dot\beta(y)\|_{\beta(y)}dy,$$
from which the statement follows.
\end{proof}

The final result we will need to relate the minimum norm and the $L^1$-norm of a test-configuration is as follows.

\begin{proposition}\cite[Proposition 5.5]{Dar-Rub}\label{J-and-L1} There exists a constant $C>1$ such that for all $\varphi\in\scH_{\omega}$ we have $$\frac{1}{C}J(\theta_{\tau})-C \leq d_1(0,\theta_{\tau})  -\frac{\AM(\theta_{\tau})}{n+1}.$$\end{proposition}
\begin{remark} The version of the above stated in \cite[Proposition 5.5]{Dar-Rub} assumes $\AM(\theta_{\tau}) = 0$, but if one does not make this assumption then the same proof gives the statement above. \end{remark}

We can now prove the following analogue of Proposition \ref{J-and-L1} for test-conf\-igurations, which relates the minimum norm to the $L^1$ norm.

\begin{proposition}\label{L1DominatesMinimum} Fix a test-configuration $(X,\co)$ and normalise the test configuration in such a way that $[\co]^{n+1}=0$. Suppose $(X,\co)$ has well-defined $L^1$-norm. Then there exists a $C>1$, independent of $(X,\co)$, such that $$\frac{1}{C}\|(\X,\co)\|_m \leq \|(\X,\co)\|_1.$$
\end{proposition}

\begin{proof} 
We have \begin{align*}\frac{1}{C}J(\varphi_{\tau})-C &\leq d_1(0,\varphi_{\tau}) -\frac{\AM(\theta_{\tau})}{n+1}, \\ &\leq \ell_1(\gamma_{\tau}) -\frac{\AM(\theta_{\tau})}{n+1}.\end{align*} Hence using Theorem \ref{MinNormAsSlope} and then Lemma \ref{derivative-of-length}

$$ \frac{1}{C} \| (\X,\co) \|_m = \lim_{{\tau} \to \infty}\frac{d}{d{\tau}}\frac{1}{C}J(\theta_{\tau})\leq \lim_{{\tau} \to \infty}\frac{d}{d{\tau}}  \ell_1(\gamma_\tau) = \lim_{{\tau} \to \infty} 
\|\dot\theta_{\tau}\|_{\theta_{\tau}} = \|(\X,\co)\|_1,$$ where we have used that $$\lim_{\tau \to \infty} \frac{d}{d\tau}\AM(\theta_{\tau}) = [\co]^{n+1}=0.$$

\end{proof}

We can now prove the existence of a destabilising point on non-trivial test-configurations.

\begin{proposition}\label{final-existence-destabilising-point} Suppose that $\|(\X,\co)\|_m>0$. Then there exists a point $p\in X$ such that $$Ch_p(\X,\co) = \frac{[\co]^{n+1}}{(n+1)[\omega]^n} - \int_C\co>0.$$\end{proposition}

\begin{proof}  First of all we normalise such that $[\co]^{n+1}=0$. We are then seeking a point $p\in X$ such that  $\int_C\co<0$. By Lemma \ref{AubinMabuchiNormalisation}, our normalisation gives that \begin{equation}\label{AubinMabuchiZero}\lim_{\tau \to \infty}\int_X \dot\theta_{\tau}\omega_{\theta_{\tau}}^n=0.\end{equation}

We argue by contradiction and assume no such point exists. By Lemma \ref{0-dimensionalKempfNess}, we therefore have that $$\lim_{{\tau}\to\infty} \dot\theta_{\tau}(p)\geq 0$$ for \emph{all} $p\in X$. In particular, if we knew we could swap limits and integration, this would give
 $$\lim_{{\tau} \to \infty} \int_X |\dot\theta_{\tau}|  \omega_{\theta_{\tau}}^n = \lim_{{\tau} \to \infty} \int_X \dot\theta_{\tau}  \omega_{\theta_{\tau}}^n=0$$ 
and hence the $L^1$-norm of the test-configuration is equal to zero and we would be done.  However, not being able to justify this step we argue instead using the associated geodesic of the test-configuration (see \cite[Proposition 2.7]{bermankpoly}). This geodesic $\psi_{\tau}$ solves a homogeneous Monge-Amp\`ere equation (in the sense of currents) on $\X_{\pi^{-1}(\partial\Delta)}$ with background metric $\Omega$, so that $\psi_{t}\to \varphi_{t}$ as $t\to 1$. Then
$$\lim_{{\tau}\to\infty} \dot \psi_\tau (p)= \lim_{{\tau}\to\infty} \dot\theta_{\tau}(p)\geq 0.$$

Thus

$$\int_X  \lim_{\tau\to \infty} |\dot\psi_\tau|  \omega_{\theta_{\tau}}^n = \int_X  \lim_{\tau\to \infty} \dot\psi_\tau  \omega_{\theta_{\tau}}^n
=  \int_X  \lim_{\tau\to \infty} \dot\theta_{\tau}  \omega_{\theta_{\tau}}^n=0$$
Now convexity of the geodesic in $\tau$ implies that
$$|\dot{\psi}_\tau|\le \lim_{\tau\to\infty} |\dot{\psi}_\tau|.$$
Hence the Dominate Convergence Theorem applies to give
$$\lim_{\tau\to \infty} \int_X |\dot\psi_\tau|  \omega_{\theta_{\tau}}^n  =\int_X  \lim_{\tau\to \infty} |\dot\psi_\tau|  \omega_{\theta_{\tau}}^n =0.$$
Now $\psi_\tau-\theta_{\tau}$ is bounded over all $\tau$ so
$$ \lim_{\tau\to \infty} \int_X |\dot\psi_\tau|  \omega_{\theta_{\tau}}^n = \lim_{\tau\to \infty} \int_X |\dot\theta_\tau| \omega_{\theta_{\tau}}^n = \|(\mathcal X,\Omega)\|_{1}.$$
Thus Proposition \ref{L1DominatesMinimum} gives that the minimum norm $\| (\mathcal X,\co)\|_m=0$ which contradicts our hypothesis. 

 Hence there must be a point such that $\int_C\co<0$ as claimed. \end{proof}

\section{Related canonical metrics and stability}\label{sec:related}

The methods of the previous sections extend almost verbatim to the setting of critical points of the J-flow and twisted cscK metrics. We briefly explain the required adaptations.

\begin{definition}[J-flow] Fix a K\"ahler metric $\alpha$ in an arbitrary K\"ahler class. A \emph{critical point of the J-flow} in a K\"ahler class $[\omega]$ is defined to be a solution to $$\Lambda_{\omega}\alpha = \gamma,$$ where the topological constant $\gamma$ is given as $$\gamma = \frac{[\alpha].[\omega]^{n-1}}{[\omega]^n},$$ and $\omega \in [\omega]$.\end{definition}

The J-flow was introduced by Donaldson \cite{Donaldson-Momentmaps}, and has close links to properties of the Mabuchi functional \cite{Chen-IMRN}. The primary analytic result of interest to us related to the J-flow is as follows.

\begin{theorem}\label{functionalsforJ-flow}\cite{Chen-IMRN,CollinsSz,SongWein} There exists a critical point of the J-flow in the K\"ahler class $[\omega]$ if and only if the functional $$\mathcal{J}_{\alpha}(\varphi) = L_{\alpha}(\varphi) - \frac{n}{n+1}\gamma\AM(\varphi)$$ is coercive. \end{theorem}

The corresponding notion of stability is as follows. 

\begin{definition}[J-stability]\cite{LejmiSz} We define the \emph{J-weight} of a test-configuration $(\X,\co)$ to be the numerical invariant $$J_{[\alpha]}(\X,\co) = [\co]^n.[q^*\alpha] - \frac{n}{n+1}\gamma [\co]^{n+1},$$ computed on a resolution of indeterminacy. We say that $(X,[\omega])$ is \emph{uniformly J-stable} if for all test-configurations we have $$J_{[\alpha]}(\X,\co) > \epsilon \|(\X,\co)\|_m,$$ for some $\epsilon$ independent of $(\X,\co)$\end{definition}

In the projective case, this definition agrees with that of Lejmi-Sz\'ekelyhidi \cite{LejmiSz} by \cite[Proposition 4.29]{DervanKeller}. The results of the previous sections extend to this setting in a trivial way to give the following, 

\begin{theorem} Fix a test-configuration $(\X,\co)$.  We have  $$\lim_{\tau \to \infty} \frac{d}{d\tau}\mathcal{J}_{\alpha}(\theta_{\tau}) = J_{[\alpha]}(\X,\co),$$ where $\rho(t)^*\co_t =\co_1 + i \ddbar \varphi_t$, $\tau = -\log |t|^2$ and $\theta_{\tau} = \varphi_t$.\end{theorem}

The following corollary is due to Lejmi-Sz\'ekelyhidi in the projective case, using different methods.

\begin{corollary} If $(X,[\omega])$ admits a critical point of the J-flow, then $(X,[\omega])$ is uniformly J-stable.\end{corollary}

An alternative proof of J-semistability in the projective case was given by the first author and Keller \cite{DervanKeller}.

A similar result holds for twisted cscK metrics, which are an important tool in constructing cscK metrics, for example through Fine's construction of cscK metrics on fibrations \cite{Fine1,Fine2} and Datar-Sz\'ekelyhidi's proof of the existence of K\"ahler-Einstein metrics on equivariantly K-stable Fano manifolds \cite{DatarSz}.

\begin{definition}[Twisted cscK metric] A \emph{twisted cscK metric} is a solution of $$S(\omega) - \Lambda_{\omega}\alpha = \hat\gamma,$$ where $\hat\gamma$ is the appropriate topological constant and $\alpha$ is an arbitrary K\"ahler metric.\end{definition}

The analogous result to Theorem \ref{functionalsforJ-flow} is as follows, and gives a slight strengthening of a result due to Berman-Berndtsson \cite{BB}.

\begin{theorem}\label{twisted-Mabuchi-coercive} If $(X,[\omega],\alpha)$ admits a twisted cscK metric, then the twisted Mabuchi functional $$\widehat\scM(\varphi) = \scM(\varphi) + \mathcal{J}_{\alpha}(\varphi)$$  is coercive. \end{theorem}

\begin{proof} 

By \cite{BB}, if $(X,[\omega],\alpha)$ admits a twisted cscK metric, then the twisted Mabuchi functional is bounded.

A straightforward argument now proves coercivity. Given a solution to $$S(\omega) - \Lambda_{\omega}\alpha = \hat\gamma,$$ one automatically obtains a solution $$S(\omega) - \Lambda_{\omega}(\alpha-\epsilon\omega) = \hat\gamma'.$$ Taking $\epsilon$ small enough so that $\alpha-\epsilon\omega$ is still K\"ahler, we see that $(X,[\omega],\alpha-\epsilon\omega)$ has bounded twisted Mabuchi functional. That is, the functional $$\widehat\scM_{\alpha-\epsilon\omega}(\varphi) = \scM(\varphi) + \mathcal{J}_{\alpha-\epsilon\omega}(\varphi)$$ is bounded below. One easily sees that $$\mathcal{J}_{\alpha-\epsilon\omega}(\varphi) = \mathcal{J}_{\alpha}(\varphi) - \epsilon (I(\varphi) - J(\varphi)),$$ so that $$\widehat\scM_{\alpha}(\varphi) \geq \epsilon(I(\varphi) - J(\varphi)) + c.$$ That is, using Lemma \ref{IJarenonneg}, the twisted Mabuchi functional is coercive.
 \end{proof}
 
 \begin{remark} In proving Corollary \ref{twisted-Mabuchi-coercive}, one could instead use the openness of solutions of the twisted cscK equation due to Hashimoto and Chen \cite{Chen-twisted,Hashimoto}. The perturbation trick used above is analogous to one used by the first author in studying twisted cscK metrics on projective varieties \cite[Lemma 3.2]{Dervan}.
 \end{remark}

Defining the twisted Donaldson-Futaki invariant of a test-configuration as $$\widehat\DF(\X,\co) = \DF(\X,\co) + J_{[\alpha]}(\X,\co)$$ and taking the obvious definition of uniform twisted K-stability, we immediately obtain the following. 

\begin{corollary} If $(X,[\omega],\alpha)$ admits a twisted cscK metric, then it is uniformly twisted K-stable. \end{corollary}

This result is due to the first author in the projective case \cite{Dervan}. A weaker result which holds also in the K\"ahler setting is due to Stoppa \cite{StoppaSlope}.


\vspace{4mm}

\end{document}